\documentclass[final,3p]{elsarticle}
\usepackage{color}
\usepackage{amssymb}
\usepackage{amsmath}
\usepackage{booktabs}
\usepackage{multirow}
\usepackage{tabularx}
\usepackage{tabu}
\usepackage{stmaryrd}
\usepackage{epstopdf}
\usepackage{subfigure} 
\usepackage{graphicx}
\usepackage{caption}
\usepackage{amsthm} 
\usepackage{amssymb} 
\usepackage{mathrsfs}

\newcommand{\mbf}{\mathbf}

\numberwithin{equation}{section}
\newtheorem{thm}{Theorem}[section]

\newtheorem{lem}[thm]{Lemma}

\newtheorem{rem}[thm]{Remark}
\newtheorem{ex}[thm]{Example}
\allowdisplaybreaks[4]

\biboptions{}
 \journal{*}
\begin{document}
	
\begin{frontmatter}
\title{Novel structure-preserving schemes for stochastic Klein--Gordon--Schr\"odinger equations with additive noise}

\author[ad1,ad2]{Jialin Hong}
\ead{hjl@lsec.cc.ac.cn} 

\author[ad3]{Baohui Hou}
\ead{houbaohui@shu.edu.cn} 

\author[ad4]{Liying Sun\corref{cof1}}
\ead{liyingsun@lsec.cc.ac.cn} 

\author[ad1]{Xiaojing Zhang}
\ead{zxjing@lsec.cc.ac.cn} 

\address[ad1]{Institute of Computational Mathematics and Scientific/Engineering Computing, Academy of Mathematics and Systems Science, Chinese Academy of Sciences, Beijing 100190, China}
\address[ad2]{School of Mathematical Sciences, University of Chinese Academy of Sciences, Beijing 100049, China}
\address[ad3]{Department of Mathematics, Shanghai University, Shanghai 200444, P.R.China}
\cortext[cof1]{Corresponding author}
\address[ad4]{Capital Normal University, Beijing 100048, P.R.China}

\begin{abstract}
Stochastic Klein--Gordon--Schr\"odinger  (KGS) equations are important mathematical models and describe the interaction between scalar nucleons and neutral scalar mesons in the stochastic environment. 
In this paper, we propose novel structure-preserving schemes to numerically solve stochastic KGS  equations with additive noise, which preserve averaged charge evolution law, averaged energy evolution law, symplecticity, and multi-symplecticity. 
By applying central difference, sine pseudo-spectral method, or finite element method in space and modifying finite difference in time, we present some charge and energy preserved fully-discrete scheme for the original system.  
In addition, combining the symplectic Runge-Kutta method in time and finite difference in space, we propose a class of multi-symplectic discretizations preserving the geometric structure of the stochastic KGS equation. 
Finally, numerical experiments confirm theoretical findings.
\end{abstract}	

	
\end{frontmatter}

\section{Introduction}

The KGS equation
\begin{equation}\label{KGS0}
\begin{cases}
\boldsymbol{\mathrm{i} }d\varphi+\left(\varphi_{x x}+\sigma\varphi u \right)dt= 0, & (x,t) \in\mathcal O \times (0, T], \\
du_{t}-\left(u_{x x}-\mu^2u +\sigma|\varphi|^2\right)dt = 0, & (x,t) \in \mathcal O \times (0, T],\\
\varphi(0) = \varphi_0(x), u(0) = u_0(x), u_t(0) = v_0(x), & x \in \mathcal O,\\
\varphi(t) = 0, u(t) = 0,  & x \in \partial\mathcal O, t\in (0, T],
\end{cases}
\end{equation}
where $\varphi$ and $u$ represent a complex scalar nucleon field and a real meson field respectively, $\mu$ is mass of a meson and $\sigma$ is a coupling real number, 
was first proposed by Isamu Fukuda and Masayoshi in 1975.  
The equation has charge and energy conservation law and models the interaction of scalar nucleons interacting with neutral scalar mesons. 
Besides, the dynamics of these fields through Yukawa coupling has been extensively studied and applied in recent decades. 

Recently, the stochastic KGS system has been widely concerned, since random effects are needed to take into account when stochasticity occurs from disturbances in the Klein-Gordon equation (see e.g., \cite{Guo1, Gao, Lu} and reference therein), and external perturbation, boundary input, and medium changing (see e.g., \cite{Lv, Yan, ZhaoLi} and references therein). 
Similar to the deterministic case, the existence of local and global solutions of stochastic KGS systems can be obtained by a priori estimates in different energy spaces. 
The averaged charge and energy of the stochastic KGS equation, as important tools, are not invariant but possess the evolution law. 
In addition to the physical characteristics, the stochastic KGS equation possesses stochastic symplecticity and multi-symplecticity, which are geometric properties. 
Constructing numerical methods preserving intrinsic structures and characters of the original system is always an important topic. 
However, there has been no result on  structure-preserving schemes for stochastic KGS equations till now.

This paper aims to design structure-preserving numerical schemes for stochastic KGS equations with additive noise.  
Inspired by the simplicity and easy programmability of the central difference (see \cite{Cui, HongWang}),  flexibility to achieve high-order accuracy and deal with complex computational domains of the finite element method (see \cite{Cohen}),  good stability and rapid convergence accuracy in solving
smooth problems of the sine pseudo-spectral method (see \cite{Hesthaven,Thalhammer}), we employ these three classic numerical methods to approximate the stochastic KGS equation. 
The corresponding three semi-discrete schemes preserve both the symplectic and multi-symplectic geometric structure, as well as the averaged charge and energy evolution law. 
When constructing fully-discrete schemes preserving both the averaged charge and energy evolution law, the treatment of the time approximation is of vital importance. 
For example, the fully-discrete scheme based on the discrete gradient method in time and central difference in space,  inheriting the energy conservation law in the deterministic case, does not preserve averaged charge and energy evolution law of the stochastic KGS system.  
Moreover, if we make use of the fully-discrete scheme preserving both charge and energy conservation law simultaneously in the deterministic case directly, neither the averaged charge nor energy evolution law of the stochastic KGS system is preserved. 
To overcome the difficulty brought by nonlinear coefficients relying on the interaction between $\varphi$ and $u$ and the coupling effect between nonlinear coefficients and driving stochastic processes, we propose some novel  averaged charge  and energy preserved fully-discrete schemes by introducing some modified terms depending on the increment of Wiener processes. 
Since the energy-preserving method can not preserve symplecticity and multi-symplecticity in general, we also employ the symplectic Runge--Kutta method, especially the parametric symplectic Runge-Kutta methods, to present various symplectic semi-discretization in time. 
Combining with the finite difference in the spatial direction, we proposed fully-discrete schemes, which satisfy the stochastic multi-symplectic conservation law.   
Finally, numerical experiments are performed to verify the theoretical result of the stochastic KGS equation.

The paper is organized as follows. In Section 2, we introduce the intrinsic properties of the stochastic KGS equations with additive noise.
In Section 3, we present fully-discrete schemes preserving discrete averaged energy and charge evolution law based on the central difference,  sine pseudo-spectral method, or finite element method in space. 
In Section 4, a class of symplectic  Runge--Kutta methods and finite difference are utilized to propose fully-discrete schemes preserving multi-symplecticity. 
Finally, numerical experiments are carried out in Section 5.

Some notations to be used:
\vspace{-0.5em}
\begin{itemize}
\item $\phi_t$, $\phi_x$  the derivative of $\phi$ with respect to time and space respectively;
\vspace{-0.5em}
\item $\phi_{xx}$  the second derivative of $\phi$ with respect to  space ;
\vspace{-0.5em}
\item $|\phi|$ the module of complex-valued function $\phi$;
\vspace{-0.5em}
\item $L^p([a,b])$ the space consisting of $p$-square integrable complex-valued functions defined on $[a,b]$ with norm $\|\cdot\|_{L^p}$;
\vspace{-0.5em}
\item $H:=L^2([a,b])$ with innner product
$\langle \phi, \psi \rangle=\operatorname{Re}(\int_{a}^b \phi(x) \bar{\psi}(x) dx)$ for $\phi, \psi \in H$ and norm $\|\cdot\|:=\|\cdot\|_{L^2}$;
\vspace{-1.5em}
\item $H^m:=H^m([a,b])$ the usual Sobolev spaces with norm  $\|\cdot\|_m$ for $m>0$;
\vspace{-0.5em}
\item $H_0^1:=H_0^1([a,b])$ the usual Sobolev spaces with norm  $\|\cdot\|_1$ and with homogeneous Dirichlet boundary condition which means $H_0^1=\{\phi|\phi(a)=\phi(b)=0,\phi\in H^1\}$;
\vspace{-0.5em}
\item $\big(\Omega, \mathcal{F},\left\{\mathcal{F}_t\right\}_{t \geq 0}, \mathbb{P}\big)$ a filtered complete probability space;
\vspace{-0.5em}
\item $\mathbb E$ the expectation operator.
\vspace{-0.5em}
\end{itemize}

\section{Intrinsic properties of stochastic KGS equation} 
Consider stochastic KGS equation equipped with the Dirichlet boundary condition as follows
\begin{equation}\label{KGS}
\left\{
\begin{aligned}
&\boldsymbol{\mathrm{i} }d\varphi+ \left(\varphi_{x x}+\varphi u \right)dt= C_1dW(t), & (x,t) \in (a, b) \times (0, T], \\
&du_{t}-\left(u_{x x}-u +|\varphi|^2\right)dt = C_2d\widetilde{W}(t), & (x,t) \in (a, b) \times (0, T],\\
&\varphi(0,x) = \varphi_0(x), u(0,x) = u_0(x), u_t(0,x) = \mu_0(x), & x \in (a, b),
\end{aligned}
\right.
\end{equation}
where $\boldsymbol{\mathrm{i} }^2=-1$, $a, b, C_1,C_{2}\in\mathbb R,$ $u_0(x), \mu_0(x)$ are real-valued functions and $\varphi_0(x)$ is a complex-valued function. 
Here, $W$ and $\widetilde{W}$ are defined as
\begin{align*}
W(t,x)=W_0(t,x)+\boldsymbol{\mathrm{i}}W_1(t,x):=\eta_{1}(x)B_0(t) + \boldsymbol{\mathrm{i}}\eta_{1}(x)B_1(t),\quad \widetilde{W}(t,x)=\eta_{2}(x)B_2(t)
\end{align*}
where $B_0(t)$, $B_1(t)$ and $B_2(t)$ are independent standard Wiener processes  and $\eta_1(x), \eta_2(x)$ are sufficiently smooth real-valued functions. 
Set $v:=\frac 12 u_t$ and let $p$ and $q$ be the real and imaginary parts of $\varphi$, respectively. 
Then \eqref{KGS} has an equivalent formalization
\begin{equation}\label{SKGS}
\left\{
\begin{aligned}
&dq = \left(p_{x x} +u p\right)dt - C_1 dW_0(t), \\
&dp = -\left(q_{x x} + u q\right)dt +C_1dW_1(t), \\
&dv = \frac{1}{2}\left( u_{x x} -  u + p^2+q^2 \right)dt + \frac{1}{2}C_2 d\widetilde{W}(t), \\
&du = 2 v dt,
\end{aligned}
\right.
\end{equation}
where $p(0)=\operatorname{Re}(\varphi_0),$
$q(0)=\operatorname{Im}(\varphi_0),$ and $v(0)=\frac 12\mu_0.$ 
Assume that $\left(\varphi_0, u_0, \mu_0\right) \in \mathcal{E}_0:=H_0^1 \times H_0^1 \times H,$ and $\eta_1 \in H_0^1$, $\eta_2 \in H$, and then the stochastic KGS equation \eqref{KGS} has a unique solution in space $L^2(\Omega, C(0, T; \mathcal{E}_0))$.  
This result is obtained by a priori estimates for $(\varphi,u,u_t)$ in different energy spaces, and most discussions are similar to those in \cite{GuoLW}.  
If the size of the noises equals 0, i.e., the noise terms are eliminated, we get the
deterministic KGS equation. 
In this case, it possesses charge conservation law and energy conservation law, where
\begin{itemize}
	\item charge: $$\|\varphi\|^2= \|p\|^2 +\|q\|^2$$
	\item energy:
	$$\mathbf H(\varphi,u,v) =  2\langle u, |\varphi|^2\rangle - (\|u\|^2 + 4\|v\|^2 + \|u\|_{1}^2+2\|\varphi\|_1^2)$$
\end{itemize}
Different from the deterministic case of $C_1, C_2=0,$ the charge and energy of \eqref{KGS} are not conserved. 
Below we shall introduce the averaged charge evolution law.

\begin{lem}
Assume that $\left(\varphi_0, u_0, \mu_0\right) \in \mathcal{E}_0:=H_0^1 \times H_0^1 \times H,$ and $\eta_1 \in H_0^1$, $\eta_2 \in H$. 
Then the stochastic KGS equation \eqref{KGS} satisfies the following averaged charge evolution law. 
\begin{equation}
\mathbb{E}[\|\varphi(t)\|^2] = \mathbb{E}[\|\varphi_0\|^2] + 2C_1^2\| \eta_1\|^2t. 
\end{equation}
\end{lem}
\begin{proof}
According to the It\^o's formula, we have 
\begin{align*}
\|p(t)\|^2= &\|p(0)\|^2  - \int_{0}^t \langle 2p(s),  q_{x x}(s) + u(s) q(s)\rangle ds
+\int_{0}^t \langle 2p(s), C_1dW_1(s)\rangle+ C_1^2\| \eta_1\|^2t, \\ 
\|q(t)\|^2= &\|q(0)\|^2  + \int_{0}^t \langle 2q(s),  p_{x x}(s) + u(s) p(s)\rangle ds
-\int_{0}^t \langle 2q(s), C_1dW_0(s)\rangle+ C_1^2\| \eta_1\|^2t.     
\end{align*}
Taking the expectation and using the integration by parts yield the result.
\end{proof}
It is obvious that the averaged charge increases linearly with respect to $t.$ 
The following lemma presents the averaged energy evolution law of stochastic KGS equation \eqref{KGS}.

\begin{lem}
Assume that $\left(\varphi_0, u_0, \mu_0\right) \in \mathcal{E}_0:=H_0^1 \times H_0^1 \times H,$ and $\eta_1 \in H_0^1$, $\eta_2 \in H$. 
The averaged energy evolution law of stochastic KGS equation \eqref{KGS} meets
\begin{equation}
\mathbb{E}[\mathbf H(\varphi(t),u(t),v(t))] = \mathbb{E}[\mathbf H(\varphi_0,u_0,\mu_0)] -  C_2^2 \| \eta_2\|^2t-  4C_1^2\|\eta_1\|_1^2t+ 4C_1^2\mathbb{E}\Big[ \int_{0}^t \langle u(s), \eta_1^2\rangle ds\Big].
\end{equation}
\end{lem}
\begin{proof}
Notice that $\mathbf H(\varphi(t),u(t),v(t)) =  2\langle u(t), |\varphi(t)|^2\rangle - (\|u(t)\|^2 + 4\|v(t)\|^2 + \|u(t)\|_{1}^2+2\|\varphi\|_1^2).$ 
By the It\^o's formula, we deduce
\begin{align*}
\|u(t)\|_1^2=& \|u_0\|_1^2 + 4\int_0^t\langle \nabla u(s), \nabla v(s)\rangle ds,\\
\|u(t)\|^2=& \|u_0\|^2  + \int_{0}^t \langle 2u(s),  2v(s)\rangle ds, \\
 \|v(t)\|^2= &\|\mu_0\|^2  + \int_{0}^t \langle v(s),  u_{x x}(s) \rangle ds
 + \int_{0}^t \langle v(s),  -u(s) +|\varphi(s)|^2\rangle ds\\
&+\int_{0}^t \langle v(s), C_2dW_2(s)\rangle
+\frac{1}{4}\int_{0}^t\langle C_2\eta_2, C_2\eta_2\rangle ds.
 \end{align*}
Combining the above three equations, we obtain 
\begin{align*}
&\|u(t)\|^2 + 4\|v(t)\|^2 + \|u(t)\|_1^2 -\|u_0\|^2 - 4\|\mu_0\|^2 - \|u_0\|_1^2\\
=&4\int_{0}^t \langle v(s),  |\varphi(s)|^2\rangle ds
  + 4C_2\int_{0}^t \langle v(s), dW_2(s)\rangle
 +C_2^2\|\eta_2\|^2t.
\end{align*}
For the terms $\|\varphi(t)\|_1^2$ and $\langle u(t), |\varphi(t)|^2\rangle$, a straight calculation yields
\begin{align*}
&\|p(t)\|_1^2+ \|q(t)\|_1^2- \|p(0)\|_1^2-\|p(0)\|_1^2\\
=&-\int_{0}^t \langle 2p_x(s),  q_{xxx}(s)+u_x(s)q(s)+u(s)q_x(s) \rangle ds
+\int_{0}^t \langle 2p_x(s), C_1d(W_1)_x(s)\rangle \\
&+\int_{0}^t \langle 2q_x(s),  p_{xxx}(s)+u_x(s)p(s)+u(s)p_x(s) \rangle ds
-\int_{0}^t \langle 2q_x(s), C_1d(W_0)_x(s)\rangle \\
&+\int_{0}^t \langle C_1(\eta_1)_x, C_1(\eta_1)_x\rangle ds
+\int_{0}^t \langle C_1(\eta_1)_x, C_1(\eta_1)_x\rangle ds\\
=&2\int_0^t\langle \nabla \varphi(s), \boldsymbol{\mathrm{i}}\nabla u(s)\varphi(s) \rangle ds -2C_1\int_0^t\langle \nabla \varphi(s), \boldsymbol{\mathrm{i}}\nabla dW(s) \rangle+2C_1^2\|\eta_1\|_1^2t,
\end{align*}
 and
\begin{align*}
&\langle u(t), |\varphi(t)|^2\rangle- \langle u_0, |\varphi_0|^2\rangle
\\=&\int_{0}^t \langle 2v(s), |\varphi(s)|^2\rangle ds-\int_{0}^t \langle u(s), 2p(s)(q_{xx}(s)+u(s)q(s))\rangle ds
+ \int_{0}^t \langle u(s), 2p(s)C_1dW_1(s)\rangle\\
&+\int_{0}^t \langle u(s), 2q(s)(p_{xx}(s)+u(s)p(s)\rangle ds
-\int_{0}^t \langle u(s), 2q(s)C_1dW_0(s)\rangle +2\int_{0}^t \langle u(s), C_1^2\eta_1^2\rangle ds\\
=&\int_{0}^t \langle 2v(s), |\varphi(s)|^2\rangle ds+2\int_0^t\langle \nabla \varphi(s), \boldsymbol{\mathrm{i}}\nabla u(s)\varphi(s) \rangle ds+2C_1\int_{0}^t \langle u(s), p(s)dW_1(s)\rangle\\
&-2C_1\int_{0}^t \langle u(s), q(s)dW_0(s)\rangle +2C_1^2\int_{0}^t \langle u(s), \eta_1^2\rangle ds.
\end{align*}
Combining the above equations leads to
\begin{align*}
&2\langle u(t), |\varphi(t)|^2\rangle - (\|u(t)\|^2 + 4\|v(t)\|^2 + \|u(t)\|_{1}^2+2\|\varphi(t)\|_1^2)\\
=&2\langle u_0, |\varphi_0|^2\rangle - (\|u_0\|^2 + 4\|\mu_0\|^2 + \|u_0\|_{1}^2+2\|\varphi_0\|_1^2)-4C_1^2\|\eta_1\|_1^2t-C_2^2\|\eta_2\|^2t\\
 &+4C_1^2\int_{0}^t \langle u(s), \eta_1^2\rangle ds- 4\int_{0}^t \langle v(s), C_2dW_2(s)\rangle
+4C_1\int_0^t\langle \nabla \varphi(s), \boldsymbol{\mathrm{i}}\nabla dW(s) \rangle\\
&+4C_1\int_{0}^t \langle u(s), p(s)dW_1(s)\rangle 
-4C_1\int_{0}^t \langle u(s), q(s)dW_0(s)\rangle.
\end{align*}
Taking expectation completes the proof.
\end{proof}

The stochastic KGS equation can also be rewritten as an infinite-dimensional stochastic Hamiltonian system. 
In detail, denoting 
\begin{align*}
&\mathbb H(p,q,u,v) =\frac 12\langle u, |\varphi|^2\rangle - \frac 14(\|u\|^2 + 4\|v\|^2 + \|u\|_{1}^2+2\|p\|_1^2+2\|q\|_1^2),\\
&\mathbb H_0(p,q,u,v):=-C_1\int_a^bp  dx,\quad \mathbb H_1(p,q,u,v):=-C_1\int_a^bq dx,\quad \mathbb H_2(p,q,u,v):=\frac{C_2}{2}\int_a^bu dx,
\end{align*}
we have 
\begin{equation}
\label{idshs}
\left\{
\begin{aligned}
& dp=-\frac{\delta \mathbb H}{\delta q}dt-\frac{\delta \mathbb H_1}{\delta q}dW_1(t),\quad &p(0)=\operatorname{Re}(\varphi_0),\\
& du=-\frac{\delta \mathbb H}{\delta v}dt,\quad&u(0)=u_0,\\
& dq=\frac{\delta \mathbb H}{\delta p}dt+\frac{\delta \mathbb H_0}{\delta p}dW_0(t),\quad&q(0)=\operatorname{Im}(\varphi_0),\\
& dv=\frac{\delta \mathbb H}{\delta u}dt+\frac{\delta \mathbb H_2}{\delta u}d\widetilde{W}(t),\quad &v(0)=\frac 12\mu_0.
\end{aligned}
\right.
\end{equation}
One of the inherent canonical properties
of the infinite-dimensional stochastic Hamiltonian system is the infinite-dimensional symplecticity of its flow. 
For \eqref{KGS} or \eqref{idshs}, the associated  symplectic form is given by
$$ \bar\omega (t) = \int_a^b (\mathrm{d} q(t) \wedge \mathrm{d} p(t) + \mathrm{d} v(t) \wedge \mathrm{d} u(t)) d x,$$
where the overbar on $\omega$ is a reminder that differential two-forms $\mathrm{~d} q \wedge \mathrm{d} p$ and $ \mathrm{d} v \wedge \mathrm{d} u$ are integrated over the
space, and ${\rm d}$ denotes the differential with respect to the initial value.

\begin{lem}
	Assume $\left(\varphi_0, u_0, \mu_0\right) \in \mathcal{E}_0:=H_0^1 \times H_0^1 \times H,$ $\eta_1 \in H_0^1$, $\eta_2 \in H,$ and that the solution of stochastic KGS equation \eqref{KGS} is differentiable with respect to the initial data. 
	Then \eqref{KGS} satisfies the infinite-dimensional stochastic symplectic structure, i.e.,
	\begin{equation*}
	\bar\omega (t) = \bar\omega (0) :=\int_a^b (\mathrm{d} q(0) \wedge \mathrm{d} p(0) + \mathrm{d} \mu_0 \wedge \mathrm{d} u_0) d x.
	\end{equation*}
\end{lem}

The lemma implies that
the spatial integral of the oriented areas of projections onto the coordinate planes is an integral invariant. 
As shown above, \eqref{KGS} is regarded as a stochastic evolution equation in time. When the spatial
variable is also of interest, both the stochastic multi-symplectic Hamiltonian system and stochastic
multi-symplectic structure are involved.

\begin{lem}
	The stochastic KGS equation \eqref{KGS} satisfies the stochastic multi-symplectic conservation law.
\end{lem} 
\begin{proof}
	
	Let $\varphi(t) = p(t) + \boldsymbol{\mathrm{i} } q(t), \varphi_x(t) = f(t) +\boldsymbol{\mathrm{i} }g(t)$ and set $r := u_t, w: = u_x$.
	Then the stochastic KGS equation  \eqref{KGS} can be reformulated as
	\begin{equation}\label{multi-symplectic}
	K d\mbf z+ L\mbf z_{x} d t =  \nabla S({\bf z})  d t + \nabla S_0({\bf z})dW_0(t)+ \nabla S_1({\bf z})dW_1(t)+ \nabla S_2({\bf z})d\widetilde{W}(t),
	\end{equation}
	where $\mbf z = (p, q, f, g, u, r, w)^\top$
	and 
	\begin{align*}
	&S({\bf z})= -\frac{1}{2}u(p^2+q^2)- \frac{1}{2}(f^2+g^2) + \frac{1}{4}(u^2+r^2-w^2),\quad \\
	&S_0({\bf z}) = C_1p,\quad ~S_1({\bf z}) =C_1q,\quad ~
	S_2({\bf z}) = - \frac{1}{2}C_2u,\\
	&K=\left[\begin{array}{ccccccc}
	0 & -1 &0&0&0&0&0\\
	1 & 0 &0&0&0&0&0\\
	0 & 0 &0&0&0&0&0\\
	0 & 0 &0&0&0&0&0\\
	0 & 0 &0&0&0&-\frac{1}{2}&0\\
	0 & 0 &0&0&\frac{1}{2}&0&0\\
	0 & 0 &0&0&0&0&0
	\end{array}\right], 
	\quad L=\left[\begin{array}{ccccccc}
	0 & 0&1&0&0&0&0\\
	0 & 0 &0&1&0&0&0\\
	-1& 0 &0&0&0&0&0\\
	0 & -1 &0&0&0&0&0\\
	0 & 0 &0&0&0&0&\frac{1}{2}\\
	0 & 0 &0&0&0&0&0\\
	0 & 0 &0&0&-\frac{1}{2}&0&0
	\end{array}\right]. 
	\end{align*} 
	From \eqref{multi-symplectic} it follows that \eqref{KGS} possesses the stochastic multi-symplectic conservation law locally
	\begin{equation}
	d_{t} \left(\mathrm{d} \mbf z \wedge K \mathrm{d} \mbf z\right) +\partial_{x}\left(\mathrm{d} \mbf z \wedge L \mathrm{d} \mbf z\right) = 0, \quad  a.s., 
	\end{equation}
	equivalently,
	\begin{align*}
	d(2\mathrm{d} q \wedge \mathrm{d} p + \mathrm{d} r \wedge \mathrm{d} u ) +\partial_{x}( 2\mathrm{d} p \wedge \mathrm{d} f + 2\mathrm{d} q \wedge \mathrm{d} g + \mathrm{d} u \wedge \mathrm{d} w)dt =0,\quad{a.s.,}
	\end{align*} 
	which implies the result. 
\end{proof}

As shown above, the stochastic Klein--Gordon--Schr\"odinger equation possesses the infinite-dimensional stochastic symplectic structure, stochastic multi-symplectic conservation law, averaged charge evolution law, and averaged energy evolution law.  
Now we introduce the fully-discrete schemes inheriting the properties of the original system.

\section{Fully-discrete schemes preserving averaged energy and charge evolution law}

In this section, we introduce fully-discrete schemes, which preserve both the averaged charge evolution law and energy evolution law of \eqref{KGS}. 

For spatial discretization, we first introduce a uniform partition with $x_i = a +ih, 0\leq i\leq M$, where $M$ is a positive integer, $\Omega_h=\left\{x_i \mid  1 \leq i \leq M-1\right\},$ $I_i = (x_{i}, x_{i+1})$ and $h=(b-a)/M$ denotes the spatial step size. 
Denoting the approximations of $p(x_i,t), q(x_i,t),v(x_i,t),u(x_i,t)$ at $x_i\in \Omega_h$ by $P_i(t), Q_i(t),V_i(t),U_i(t)$  and making use of the central difference, we have 
\begin{align}\label{semi}
\begin{cases}
dQ_i(t) = \left( \delta_x^2P_i(t) +U_i(t)\cdot P_i(t)\right)dt - C_1\eta_1(x_i)dB_0(t), \\
dP_i(t)= -\left(\delta_x^2Q_i(t) + U_i(t)\cdot Q_i(t)\right)dt + C_1\eta_1(x_i)dB_1(t), \\
dV_i(t) = \frac{1}{2} \left(\delta_x^2U_i(t) - U_i(t) +P_i(t)\cdot P_i(t)+Q_i(t)\cdot Q_i(t)\right)dt + \frac{1}{2}C_2\eta_2(x_i)dB_2(t), \\
dU_i(t) = 2 V_i(t) dt.
\end{cases}
\end{align}
where $\delta_x^2P_i= \left(P_{i-1} - 2P_i +P_{i+1}\right)/h^2,$  $\delta_x^2Q_i= \left(Q_{i-1} - 2Q_i +Q_{i+1}\right)/h^2,$ and $\delta_x^2U_i= \left(U_{i-1} - 2U_i +U_{i+1}\right)/h^2$ for $i\in\{1,\ldots,M-1\}$ approximate $p_{xx}$, $q_{xx}$ and $u_{xx}$ at $x_i\in \Omega_h$ , respectively. 
Let
\begin{equation*}\begin{aligned}
&\mbf P_M= \left(P_1, P_2, \dots,P_{M-1}\right)^\top,\quad\mbf Q_M = \left(Q_1, Q_2, \dots,Q_{M-1}\right)^\top,\\
&\mbf U_M= \left(U_1, U_2, \dots,U_{M-1}\right)^\top,\quad\mbf V_M = \left(V_1, V_2, \dots,V_{M-1}\right)^\top,\\
&\boldsymbol{\eta}_1=\left(\eta_1(x_1), \eta_1(x_2), \dots,\eta_1(x_{M-1})\right)^\top,\quad \boldsymbol{\eta}_2=\left(\eta_2(x_1), \eta_2(x_2), \dots,\eta_2(x_{M-1})\right)^\top,
\end{aligned} \end{equation*}
and \begin{equation}
\mbf{A}=\frac{1}{h^2}{
	\left[ \begin{array}{cccccccc}
	-2&1&0&0&\cdots&0&0&0\\[0.05in]
	1&-2&1&0&\cdots&0&0&0\\
	\vdots&\vdots&\vdots&\vdots&&\vdots&\vdots&\vdots\\
	0&0&0&0&\cdots&1&-2&1\\[0.05in]
	0&0&0&0&\cdots&0&1&-2
	\end{array}
	\right ]}_{(M-1)\times (M-1)}.
\end{equation}	
We obtain the semi-discretization \eqref{semi} in matrix-vector form as follows
\begin{equation}
\label{semi-discrete}
\left\{
\begin{aligned}
&d\mbf Q_M(t) = \left( \mbf A\mbf P_M(t) +\mbf U_M(t)\cdot \mbf P_M(t)\right)dt - C_1\boldsymbol{\eta}_1dB_0(t), \\
&d\mbf P_M(t)= -\left(\mbf A\mbf Q_M(t) + \mbf U_M(t)\cdot \mbf Q_M(t)\right)dt + C_1\boldsymbol{\eta}_1dB_1(t), \\
&d\mbf V_M(t) = \frac{1}{2} \left(\mbf A\mbf U_M(t) - \mbf U_M(t) +\mbf P_M(t)\cdot \mbf P_M(t)+\mbf Q_M(t)\cdot \mbf Q_M(t)\right)dt + \frac{1}{2}C_2\boldsymbol{\eta}_2dB_2(t), \\
&d\mbf U_M(t) = 2 \mbf V_M(t) dt.
\end{aligned}
\right.
\end{equation}
Here, $\mbf U_M(t)\cdot \mbf P_M(t)$ denotes the components multiplication one by one between $\mbf U_M(t)$ and $\mbf P_M(t)$, the same as $\mbf U_M(t)\cdot \mbf Q_M(t),\mbf P_M(t)\cdot \mbf P_M(t),\mbf Q_M(t)\cdot \mbf Q_M(t)$.
Denote the inner products of discrete Hilbert space $l^2_h$ by 
$\langle f,g\rangle_h=h\sum\limits_{i=1}^{M-1}\operatorname{Re}(f_i\bar g_i)$ and $\|f\|_h=\sqrt{\langle f,f\rangle_h}.$ 
As a result, the charge and energy of \eqref{semi-discrete} read 
\begin{align*}
&\mathcal{N}(\mbf P_M,\mbf Q_M) = \|\mbf P_M\|_h^2 +\|\mbf Q_M\|_h^2,\\
&\mathcal{H}(\mbf P_M,\mbf Q_M,\mbf U_M,\mbf V_M) =  2\langle \mbf U_M, \mbf P_M\cdot \mbf P_M+ \mbf Q_M\cdot \mbf Q_M\rangle_h- \| \mbf U_M\|_h^2 - 4\| \mbf V_M\|_h^2 \\
&\quad\quad\quad\quad\quad\quad\quad\quad\quad\quad\quad\quad+ \langle \mbf U_M, \mbf A\mbf U_M\rangle_h+ 2  \langle \mbf P_M, \mbf A\mbf P_M\rangle_h +  2\langle \mbf Q_M, \mbf A\mbf Q_M\rangle_h,
\end{align*}
respectively. 
The central difference \eqref{semi-discrete} preserves the averaged charge and energy evolution law, which are shown in the following theorems. 

\begin{thm} \label{thm_charge}
	Assume $\left(\varphi_0, u_0, \mu_0\right) \in \mathcal{E}_0:=H_0^1 \times H_0^1 \times H,$ $\eta_1 \in H_0^1$, $\eta_2 \in H.$ 
	The averaged charge for semi-discretization \eqref{semi-discrete} has the following evolutionary relationship
	\begin{equation}\label{eq3.8}
	\mathbb{E}[\mathcal{N}(\mbf P_M(t),\mbf Q_M(t))  ] = \mathbb{E}[\mathcal{N}(\mbf P_M(0),\mbf Q_M(0))  ]  + 2C_1^2\|\boldsymbol{\eta}_1\|_h^2t.
	\end{equation}
\end{thm}
\begin{proof}
	By the  It\^o's  formula for $\|\mbf P_M(t)\|_h^2$ and $\|\mbf Q_M(t)\|_h^2$, respectively, we deduce
	\begin{align*}
	\label{eq3.9}
	\|\mbf P_M(t)\|_h^2= &\|\mbf P_M(0)\|_h^2  - \int_{0}^t \langle 2\mbf P_M(s),  \mbf A\mbf Q_M(s) + \mbf U_M(s)\cdot \mbf Q_M(s)\rangle_h ds +\int_{0}^t \langle 2\mbf P_M(s), C_1\boldsymbol{\eta}_1\rangle_hdB_1(s)\\
	&+C_1^2\|\boldsymbol{\eta}_1\|_h^2t,\\
	\|\mbf Q_M(t)\|_h^2= &\|\mbf Q_M(0)\|_h^2  + \int_{0}^t \langle 2\mbf Q_M(s), \mbf A\mbf P_M(s) +\mbf U_M(s)\cdot \mbf P_M(s)\rangle_h ds
	-\int_{0}^t \langle 2\mbf Q_M(s), C_1\boldsymbol{\eta}_1\rangle_hdB_0(s)\\
	&+C_1^2\|\boldsymbol{\eta}_1\|_h^2t.     
	\end{align*}
	Making use of the symmetric property of matrix $ \mbf A$ i.e., $\langle \mbf P_M(s), \mbf A\mbf Q_M(s) \rangle_h = \langle \mbf Q_M(s),  \mbf A\mbf P_M(s) \rangle_h$ and  taking the expectation yield the result.
\end{proof}

\begin{thm} \label{thm_energy}
	Assume $\left(\varphi_0, u_0, \mu_0\right) \in \mathcal{E}_0:=H_0^1 \times H_0^1 \times H,$ $\eta_1 \in H_0^1$, $\eta_2 \in H.$ 
	The semi-discretization \eqref{semi-discrete} has the following averaged energy evolution law
	\begin{align}
	&\mathbb{E}[\mathcal{H}(\mbf P_M(t),\mbf Q_M(t),\mbf U_M(t),\mbf V_M(t))]\nonumber\\
	=&\mathbb{E}[ \mathcal{H}(\mbf P_M(0),\mbf Q_M(0),\mbf U_M(0),\mbf V_M(0))]- C_2^2 \mathcal{Q}_2 t+ 4C_1^2\widetilde{\mathcal{Q}}_1t+ 4C_1^2\mathbb{E}\Big[\int_{0}^t \langle \mbf U_M(s), \boldsymbol{\eta}_1\cdot \boldsymbol{\eta}_1\rangle_h ds\Big], 
	\end{align}
	where  $\mathcal{Q}_2=\|\boldsymbol{\eta}_2\|_h^2, ~\widetilde{\mathcal{Q}}_1 = \langle \boldsymbol{\eta}_1, \mbf A\boldsymbol{\eta}_1\rangle_h$.
\end{thm}
\begin{proof}
	Applying the It\^o's formula to $\| \mbf U_M(t)\|_h^2$ and $\| \mbf V_M(t)\|_h^2$, we obtain
	\begin{align*}
	\| \mbf U_M(t)\|_h^2 = &\| \mbf U_M(0)\|_h^2 + 4\int_{0}^t \langle \mbf U_M(t),  \mbf V_M(t)\rangle_h ds, \\
	\|\mbf V_M(t)\|_h^2= &\|\mbf V_M(0)\|_h^2  
	+ \int_{0}^t \langle \mbf V_M(s),  -\mbf U_M(s)  +\mbf P_M(s)\cdot \mbf P_M(s)+\mbf Q_M(s)\cdot \mbf Q_M(s)\rangle_h ds\\
	&+ \int_{0}^t \langle \mbf V_M(s),  A\mbf U_M(s)  \rangle_h ds+\int_{0}^t \langle \mbf V_M(s), C_2\boldsymbol{\eta}_2\rangle_hdB_2(t)
	+\frac{1}{4}C_2^2\|\boldsymbol{\eta}_2\|_h^2t,\\
	\langle \mbf U_M(t),  A\mbf U_M(t)  \rangle_h=&\langle \mbf U_M(0),  A\mbf U_M(0)  \rangle_h+2\int_{0}^t \Big\langle \frac{d\mbf U_M(s) }{ds},  A\mbf U_M(s)  \Big\rangle_h ds\\
	=&\langle \mbf U_M(0),  A\mbf U_M(0)  \rangle_h+4\int_{0}^t \langle \mbf V_M(s),  A\mbf U_M(s)  \rangle_h ds.
	\end{align*}
	Based on the above three equations, we have
	\begin{equation}\begin{aligned} \label{eq3.13}
	&\mathbb{E}[ \| \mbf U_M(t)\|_h^2 + 4\|\mbf V_M(t)\|_h^2-\langle \mbf U_M(t),  A\mbf U_M(t)  \rangle_h ]- \mathbb{E}[\| \mbf U_M(0)\|_h^2 +  4\|\mbf V_M(0)\|_h^2 -  \langle \mbf U_M(0),  A\mbf U_M(0)  \rangle_h ]\\
	=&\mathbb{E}\Big[ \int_{0}^t 4\langle \mbf V_M(s), \mbf P_M(s)\cdot \mbf P_M(s)+\mbf Q_M(s)\cdot \mbf Q_M(s)\rangle_h ds \Big]
	+ C_2^2\|\boldsymbol{\eta}_2\|_h^2t.\\
	\end{aligned}\end{equation}
	A straight calculation leads to
	\begin{align*}
	&\langle \mbf U_M(t),  \mbf P_M(t)\cdot \mbf P_M(t) \rangle_h-\langle \mbf U_M(0),  \mbf P_M(0)\cdot \mbf P_M(0) \rangle_h\\
	=&\int_{0}^t 2\langle \mbf V_M(s), \mbf P_M(s) \cdot \mbf P_M(s) \rangle_h ds- \int_{0}^t 2\langle \mbf U_M(s), \mbf P_M(s) \cdot (A\mbf Q_M(s) + \mbf U_M(s)\cdot \mbf Q_M(s) ) \rangle_h ds \\
	&+2C_1\int_{0}^t\langle \mbf U_M(s),  \mbf P_M(s)\cdot \boldsymbol{\eta}_1\rangle_hdB_1(s)+ C_1^2\int_{0}^t \langle \mbf U_M(s),  \boldsymbol{\eta}_1\cdot \boldsymbol{\eta}_1 \rangle_hds.\\
	&\langle \mbf U_M(t),  \mbf Q_M(t)\cdot \mbf Q_M(t) \rangle_h-\langle \mbf U_M(0),  \mbf Q_M(0)\cdot \mbf Q_M(0) \rangle_h \\
	=&\int_{0}^t 2\langle \mbf V_M(s), \mbf Q_M(s) \cdot \mbf Q_M(s) \rangle_h ds +\int_{0}^t 2\langle \mbf U_M(s), \mbf Q_M(s) \cdot (A\mbf P_M(s) + \mbf U_M(s)\cdot \mbf P_M(s) ) \rangle_h ds\\
	&-2C_1\int_{0}^t\langle \mbf U_M(s),\mbf Q_M(s)\cdot\boldsymbol{\eta}_1\rangle_hdB_0(s)
	+C_1^2\int_{0}^t \langle\mbf U_M(s),  \boldsymbol{\eta}_1\cdot \boldsymbol{\eta}_1\rangle_hds.\\
	&\langle \mbf P_M(t),  \mbf A\mbf P_M(t)  \rangle_h -\langle \mbf P_M(0),  \mbf A\mbf P_M(0)  \rangle_h\\
	=&- \int_{0}^t  2\langle \mbf A\mbf P_M(s) ,\mbf A\mbf Q_M(s) + \mbf U_M(s)\cdot \mbf Q_M(s)\rangle_h ds+ 2C_1\int_{0}^t \langle \mbf A\mbf P_M(s) ,\boldsymbol{\eta}_1\rangle_hdB_1(s)
	+ C_1^2\int_{0}^t \langle \mbf A\boldsymbol{\eta}_1,\boldsymbol{\eta}_1\rangle_hds.\\
	&\langle \mbf Q_M(t),  \mbf A\mbf Q_M(t)  \rangle_h -\langle \mbf Q_M(0),  \mbf A\mbf Q_M(0)  \rangle_h \\
	=& \int_{0}^t  2\langle \mbf A\mbf Q_M(s) ,\mbf A\mbf P_M(s) + \mbf U_M(s)\cdot \mbf P_M(s)\rangle_h ds - 2C_1\int_{0}^t  \langle \mbf A\mbf Q_M(s) ,\boldsymbol{\eta}_1\rangle_hdB_0(s)
	+ C_1^2\int_{0}^t \langle \mbf A\boldsymbol{\eta}_1,\boldsymbol{\eta}_1\rangle_hds.
	\end{align*}
	Combining the above equations and taking expectation complete proof. 
\end{proof}

For any $T>0$, we partition the time domain $[0,T]$ uniformly with nodes $t_n = n\Delta t, n = 0,1,\ldots,N$ and $N = [T/\Delta t]$. 
The fully-discrete scheme preserving charge and energy evolution law also depends on the numerical discretization in time, which confronts the difficulty brought by the treatment of the time approximation on both drift and diffusion coefficients. 
By introducing some modified terms and
taking advantage of finite difference method to solve  \eqref{semi-discrete}, we have the following fully-discrete scheme
\begin{subequations}\label{CEFD}
\begin{align}
\label{CEFD3}
\mbf P_M^{n} = &\mbf P_M^{n-1}-\Delta t\left( \mbf A \frac{\mbf Q_M^{n}+\mbf Q_M^{n-1}}{2}+\frac{1}{2}\mbf U_M^{n}\cdot\left(\mbf Q_M^{n}+\mbf Q_M^{n-1}\right)\right)+C_1\boldsymbol{\eta}_1\Delta B_1^n\nonumber\\
&+\frac{1}{2}\Delta t C_1\Delta B_0^n\left(
A\boldsymbol{\eta}_1+\mbf U_M^n\cdot\boldsymbol{\eta}_1\right),\\
\label{CEFD4}
\mbf Q_M^{n} = &\mbf Q_M^{n-1}+\Delta t\left( \mbf A \frac{\mbf P_M^{n}+\mbf P_M^{n-1}}{2}+\frac{1}{2}\mbf U_M^{n}\cdot\left(\mbf P_M^{n}+\mbf P_M^{n-1}\right)\right)-C_1\boldsymbol{\eta}_1\Delta B_0^n\nonumber\\
&+\frac{1}{2}\Delta t C_1\Delta B_1^n\left(
A\boldsymbol{\eta}_1+\mbf U_M^n\cdot\boldsymbol{\eta}_1\right), \\
\label{CEFD5}
\mbf V_M^{n} =&\mbf V_M^{n-1}+\frac{1}{2}\Delta t\left(  \mbf A \frac{\mbf U_M^{n}+\mbf U_M^{n-1}}{2}-\frac{\mbf U_M^{n}+\mbf U_M^{n-1}}{2}+\mbf P_M^{n-1} \cdot \mbf P_M^{n-1}+\mbf Q_M^{n-1} \cdot \mbf Q_M^{n-1}\right)\\
&+\frac{1}{2}C_2\boldsymbol{\eta}_2\Delta B_2^n+\Delta t\left(C_1\mbf P_M^{n-1}\cdot\boldsymbol{\eta}_1\Delta B_1^n-C_1\mbf Q_M^{n-1}\cdot\boldsymbol{\eta}_1\Delta B_0^n+\Delta tC_1^2\boldsymbol{\eta}_1\cdot\boldsymbol{\eta}_1\right), \nonumber\\
\label{CEFD6}
\mbf U_M^{n} =&\mbf U_M^{n-1}+\Delta t\left(\mbf V_M^{n}+\mbf V_M^{n-1}\right)+\frac{1}{2}\Delta t C_2\boldsymbol{\eta}_2\Delta B_2^n,
\end{align}
\end{subequations}
where $\Delta B_0^n = B_0(t_n)-B_0(t_{n-1}), ~\Delta B_1^n = B_1(t_n)-B_1(t_{n-1}),~\Delta B_2^n = B_2(t_n)-B_2(t_{n-1})$.

\begin{thm} \label{thm5.1}
	Assume $\left(\varphi_0, u_0, \mu_0\right) \in \mathcal{E}_0:=H_0^1 \times H_0^1 \times H,$ $\eta_1 \in H_0^1$, $\eta_2 \in H.$ 
	The averaged charge for fully-discrete scheme \eqref{CEFD} has the following evolutionary relationship
	\begin{equation*}
	\mathbb{E}[\mathcal{N}(\mbf P_M^n,\mbf Q_M^n)  ] = \mathbb{E}[\mathcal{N}(\mbf P_M^0,\mbf Q_M^0)]  + 2C_1^2\|\boldsymbol{\eta}_1\|_h^2t_n,
	\end{equation*}
	where $n\in\{0,1,\ldots,N\}.$
\end{thm}
\begin{proof}
	From \eqref{CEFD} it follows that the associated one-step approximation is
	\begin{align}
	&\overline{\mbf P}_M^1=\mbf P_M^0 + C_1\Delta \mbf W_1^1, ~~\overline{\mbf Q}_M^1= \mbf Q_M^0- C_1\Delta \mbf W_0^1, ~~\overline{\mbf V}_M^1= \mbf V_M^0 +  \frac{1}{2}C_2\Delta \mbf W_2^1, ~~\overline{\mbf U}_M^1= \mbf U_M^0,\nonumber\\
	&\mbf P_M^{1} = \overline{\mbf P}_M^1-\Delta t\left( \mbf A \frac{\mbf Q_M^{1}+\overline{\mbf Q}_M^1}{2}+\frac{1}{2}\mbf U_M^{1}\cdot\left(\mbf Q_M^{1}+\overline{\mbf Q}_M^1\right)\right),\nonumber\\
	\label{CEFD-3}
	&\mbf Q_M^{1} =\overline{\mbf Q}_M^1+\Delta t\left( \mbf A \frac{\mbf P_M^{1}+\overline{\mbf P}_M^1}{2}+\frac{1}{2}\mbf U_M^{1}\cdot\left(\mbf P_M^{1}+\overline{\mbf P}_M^1\right)\right), \\
	&\mbf V_M^{1} =\overline{\mbf V}_M^1+\Delta t\left(\frac{1}{2} \mbf A \frac{\mbf U_M^{1}+\overline{\mbf U}_M^1}{2}-\frac{1}{2} \frac{\mbf U_M^{1}+\overline{\mbf U}_M^1}{2}+\frac{1}{2}\left(\overline{\mbf P}_M^1 \cdot \overline{\mbf P}_M^1+\overline{\mbf Q}_M^1 \cdot \overline{\mbf Q}_M^1\right)\right),\nonumber \\
	&\mbf U_M^{1} =\overline{\mbf U}_M^1+\Delta t\left(\mbf V_M^{1}+\overline{\mbf V}_M^1\right),\nonumber
	\end{align}
	where $\Delta{\bf W}_0^1=\boldsymbol{\eta}_1(B_0(h)-B_0(0)),$ $\Delta {\bf W}_1^1=\boldsymbol{\eta}_1(B_1(h)-B_1(0)),$ and $\Delta {\bf W}_2^1=\boldsymbol{\eta}_2(B_2(h)-B_2(0)).$ 
	Taking expectation leads to 
	\begin{align*}\label{eq5.3-1}
	\mathbb{E}[\langle \overline{\mbf P}_M^1, \overline{\mbf P}_M^1 \rangle_h + \langle \overline{\mbf Q}_M^1, \overline{\mbf Q}_M^1 \rangle_h ] = \mathbb{E}[\langle \mbf P_M^0, \mbf P_M^0 \rangle_h + \langle \mbf Q_M^0, \mbf Q_M^0 \rangle_h]  + 2C_1^2\|\boldsymbol{\eta}_1\|_h^2\Delta t,
	\end{align*} 
	Making use of \eqref{CEFD-3}, we obtain
	\begin{align*}
	&\langle \mbf P_M^1, \mbf P_M^1 \rangle_h- \langle \overline{\mbf P}_M^1, \overline{\mbf P}_M^1 \rangle_h = -\Delta t \Big\langle \mbf A \frac{\mbf Q_M^{1}+\overline{\mbf Q}_M^1}{2}+\frac{1}{2}\mbf U_M^{1}\cdot\left(\mbf Q_M^{1}+\overline{\mbf Q}_M^1\right), \mbf P_M^{1} +\overline{\mbf P}_M^1\Big\rangle_h, \\
	&\langle \mbf Q_M^1, \mbf Q_M^1 \rangle_h- \langle \overline{\mbf Q}_M^1, \overline{\mbf Q}_M^1 \rangle_h = \Delta t \Big\langle \mbf A \frac{\mbf P_M^{1}+\overline{\mbf P}_M^1}{2}+\frac{1}{2}\mbf U_M^{1}\cdot\left(\mbf P_M^{1}+\overline{\mbf P}_M^1\right), \mbf Q_M^{1} +\overline{\mbf Q}_M^1\Big\rangle_h.
	\end{align*}
	Taking advantage of the symmetric property of matrix $ \mbf A$ and taking expectation, we derive the result. 
\end{proof}

\begin{thm}  \label{thm5.2}
	Assume $\left(\varphi_0, u_0, \mu_0\right) \in \mathcal{E}_0:=H_0^1 \times H_0^1 \times H,$ $\eta_1 \in H_0^1$, $\eta_2 \in H.$  
	The averaged energy for fully-discrete scheme \eqref{CEFD} has the following evolutionary relationship
	\begin{align*}
	&\mathbb{E}[\mathcal{H}(\mbf P_M^n,\mbf Q_M^n,\mbf U_M^n,\mbf V_M^n)  ] \\
	=&\mathbb{E}[\mathcal{H}(\mbf P_M^0,\mbf Q_M^0,\mbf U_M^0,\mbf V_M^0)]
	- C_2^2 \mathcal{Q}_2 t_n+ 4C_1^2\widetilde{\mathcal{Q}}_1t_n+ 4C_1^2\sum_{i=0}^{n-1}
	\mathbb{E}[\langle \mbf U_M^i, \boldsymbol{\eta}_1\cdot \boldsymbol{\eta}_1\rangle_h]\Delta t, 
	\end{align*}
	where $n\in\{0,1,\ldots,N\},$  $\mathcal{Q}_2=\|\boldsymbol{\eta}_2\|_h^2, ~\widetilde{\mathcal{Q}}_1 = \langle \boldsymbol{\eta}_1, \mbf A\boldsymbol{\eta}_1\rangle_h$.
\end{thm}

\begin{proof}
	By the one-step approximation \eqref{CEFD-3}, 
	\begin{align*}
	& \langle \overline{\mbf U}_M^1, \overline{\mbf U}_M^1 \rangle_h = \langle \mbf U_M^0, \mbf U_M^0 \rangle_h, \quad 
	\langle \overline{\mbf U}_M^1, \mbf A\overline{\mbf U}_M^1 \rangle_h = \langle \mbf U_M^0, \mbf A\mbf U_M^0 \rangle_h,\\
	&\langle \overline{\mbf V}_M^1, \overline{\mbf V}_M^1 \rangle_h = \langle \mbf V_M^0, \mbf V_M^0 \rangle_h + \langle \overline{\mbf V}_M^0, C_2\Delta \mbf W_2^1 \rangle_h + \frac{C_2^2}{4}\mathcal{Q}_2\Delta t,
	\\
	&\langle \overline{\mbf P}_M^1, \mbf A\overline{\mbf P}_M^1 \rangle_h = \langle \mbf P_M^0, \mbf A\mbf P_M^0 \rangle_h + 2\langle \mbf A\mbf P_M^0, C_1\Delta \mbf W_1^1 \rangle_h + C_1^2 \widetilde{\mathcal{Q}}_1 \Delta t,\\
	&\langle \overline{\mbf Q}_M^1, \mbf A\overline{\mbf Q}_M^1 \rangle_h = \langle \mbf Q_M^0,  \mbf A\mbf Q_M^0 \rangle_h - 2\langle \mbf A\mbf Q_M^0, C_1\Delta \mbf W_0^1 \rangle_h + C_1^2 \widetilde{\mathcal{Q}}_1 \Delta t,\\
	\nonumber
	&\langle \overline{\mbf U}_M^1,  \overline{\mbf P}_M^{1} \cdot \overline{\mbf P}_M^{1} + \overline{\mbf Q}_M^{1} \cdot \overline{\mbf Q}_M^{1}\rangle_h
	= \langle \mbf U_M^0, \mbf P_M^0 \cdot \mbf P_M^0 +\mbf Q_M^0 \cdot \mbf Q_M^0\rangle_h - 2\langle \mbf U^0_M, C_1\mbf Q_M^0\cdot\Delta \mbf W_0^1\rangle_h\\
	&\quad\quad\quad\quad+2\langle \mbf U_M^0, C_1\mbf P_M^0\cdot\Delta \mbf W_1^1 \rangle_h
	+C_1^2\langle \mbf U_M^0, \Delta{\bf W}_0^1 \cdot \Delta{\bf W}_0^1\rangle_h
	+C_1^2\langle \mbf U_M^0, \Delta{\bf W}_1^1 \cdot \Delta{\bf W}_1^1\rangle_h.
	\end{align*}
	Further we employ the definition of $\mathcal{H}$, take expectation, and then obtain
	\begin{align*}
	&\mathbb{E}[\mathcal{H}(\overline{\mbf P}_M^{1},\overline{\mbf Q}_M^{1},\overline{\mbf U}_M^{1},\overline{\mbf V}_M^{1})  ] \\
	=&\mathbb{E}[\mathcal{H}(\mbf P_M^0,\mbf Q_M^0,\mbf U_M^0,\mbf V_M^0)] -C_2^2\mathcal{Q}_2\Delta t + 4C_1^2 \widetilde{\mathcal{Q}}_1 \Delta t+ 4C_1^2\mathbb{E}[\langle  \mbf U_M^0, \boldsymbol{\eta}_1\cdot\boldsymbol{\eta}_1 \rangle_h]\Delta t .
	\end{align*}
	Moreover,  from \eqref{CEFD-3} it follows that
	\begin{align*}
	&\langle \mbf P_M^1-\overline{\mbf P}_M^1, \mbf A(\mbf P_M^1 +\overline{\mbf P}_M^1) \rangle_h = -\Delta t \Big\langle \mbf A \frac{\mbf Q_M^{1}+\overline{\mbf Q}_M^1}{2}+\frac{1}{2}\mbf U_M^{1} \cdot\left(\mbf Q_M^{1}+\overline{\mbf Q}_M^1\right), \mbf A(\mbf P_M^{1} +\overline{\mbf P}_M^1)\Big\rangle_h, \\
	&\langle \mbf Q_M^1-\overline{\mbf Q}_M^1, \mbf A(\mbf Q_M^1 +\overline{\mbf Q}_M^1) \rangle_h= \Delta t \Big\langle \mbf A \frac{\mbf P_M^{1}+\overline{\mbf P}_M^1}{2}+\frac{1}{2}\mbf U_M^{1} \cdot\left(\mbf P_M^{1}+\overline{\mbf P}_M^1\right), \mbf A(\mbf Q_M^{1} +\overline{\mbf Q}_M^1)\Big\rangle_h,\\
	&\mbf P_M^1\cdot  \mbf P_M^1+ \mbf Q_M^1\cdot  \mbf Q_M^1 = \overline{\mbf P}_M^1\cdot \overline{\mbf P}_M^1+\overline{\mbf Q}_M^1\cdot \overline{\mbf Q}_M^1-\Delta t \mbf A \frac{\mbf Q_M^{1}+\overline{\mbf Q}_M^1}{2} \cdot (\mbf P_M^1 +\overline{\mbf P}_M^1) \\
	&\qquad\qquad\qquad\qquad\qquad\quad+\Delta t \mbf A \frac{\mbf P_M^{1}+\overline{ \mbf P}_M^1}{2} \cdot (\mbf Q_M^1 +\overline{\mbf Q}_M^1),
	\end{align*}
	and 
	\begin{align*}
	&\langle \mbf U_M^1-\overline{\mbf U}_M^1, \mbf U_M^1 +\overline{\mbf U}_M^1 \rangle_h = \langle \Delta t(\mbf V_M^{1}+\overline{\mbf V}_M^1), \mbf U_M^1 +\overline{\mbf U}_M^1 \rangle_h,\\
	&\langle \mbf U_M^1-\overline{\mbf U}_M^1, \mbf A(\mbf U_M^1 +\overline{\mbf U}_M^1) \rangle_h = \langle \Delta t(\mbf V_M^{1}+\overline{\mbf V}_M^1), \mbf A(\mbf U_M^1 +\overline{\mbf U}_M^1) \rangle_h,\\
	&\langle \mbf V_M^1-\overline{\mbf V}_M^1, \mbf V_M^1 +\overline{\mbf V}_M^1 \rangle_h\\
	 =& \frac{\Delta t}{4}\Big\langle \mbf A (\mbf U_M^{1}+\overline{\mbf U}_M^1)-(\mbf U_M^{1}+\overline{\mbf U}_M^1)+\mbf P_M^{1} \cdot \mbf P_M^{1}+\mbf Q_M^{1} \cdot \mbf Q_M^{1}+\overline{\mbf P}_M^1 \cdot \overline{\mbf P}_M^1+\overline{\mbf Q}_M^1 \cdot \overline{\mbf Q}_M^1\\
	 &+\Delta t \mbf A \frac{\mbf Q_M^{1}+\overline{\mbf Q}_M^1}{2} \cdot (\mbf P_M^1 +\overline{\mbf P}_M^1)	-\Delta t \mbf A \frac{\mbf P_M^{1}+\overline{ \mbf P}_M^1}{2} \cdot (\mbf Q_M^1 +\overline{\mbf Q}_M^1), \mbf V_M^1 +\overline{\mbf V}_M^1\Big\rangle_h.
	\end{align*}
Then we obtain
\begin{align*}
	&2\langle \mbf U_M^1, \mbf P_M^{1} \cdot \mbf P_M^{1}+\mbf Q_M^{1} \cdot \mbf Q_M^{1}\rangle_h \\
	=&\langle\mbf U_M^1+\overline{\mbf U}_M^1, \overline{\mbf P}_M^1\cdot \overline{\mbf P}_M^1+\overline{\mbf Q}_M^1\cdot \overline{\mbf Q}_M^1-\Delta t \mbf A \frac{\mbf Q_M^{1}+\overline{\mbf Q}_M^1}{2} \cdot (\mbf P_M^1 +\overline{\mbf P}_M^1)+\Delta t \mbf A \frac{\mbf P_M^{1}+\overline{ \mbf P}_M^1}{2} \cdot (\mbf Q_M^1 +\overline{\mbf Q}_M^1)\rangle_h\\
	&+\langle\Delta t(\mbf V_M^1+\overline{\mbf V}_M^1), \mbf P_M^{1} \cdot \mbf P_M^{1}+\mbf Q_M^{1} \cdot \mbf Q_M^{1}\rangle_h\\
	=&\langle\mbf U_M^1+\overline{\mbf U}_M^1, \overline{\mbf P}_M^1\cdot \overline{\mbf P}_M^1+\overline{\mbf Q}_M^1\cdot \overline{\mbf Q}_M^1\rangle_h+\langle\Delta t(\mbf V_M^1+\overline{\mbf V}_M^1), \mbf P_M^{1} \cdot \mbf P_M^{1}+\mbf Q_M^{1} \cdot \mbf Q_M^{1}\rangle_h\\
	&-\langle\mbf U_M^1+\overline{\mbf U}_M^1,\Delta t \mbf A \frac{\mbf Q_M^{1}+\overline{\mbf Q}_M^1}{2} \cdot (\mbf P_M^1 +\overline{\mbf P}_M^1)\rangle_h+\langle\mbf U_M^1+\overline{\mbf U}_M^1,\Delta t \mbf A \frac{\mbf P_M^{1}+\overline{ \mbf P}_M^1}{2} \cdot (\mbf Q_M^1 +\overline{\mbf Q}_M^1)\rangle_h.
	\end{align*}
Since $\mbf U_M^1+\overline{\mbf U}_M^1=2\mbf U_M^1+\Delta t(\mbf V_M^{1}+\overline{\mbf V}_M^1),$
	\begin{align*}
	&2\langle \mbf U_M^1, \mbf P_M^{1} \cdot \mbf P_M^{1}+\mbf Q_M^{1} \cdot \mbf Q_M^{1}\rangle_h \\
		=&\langle 2\overline{\mbf U}_M^1,\overline{\mbf P}_M^1\cdot \overline{\mbf P}_M^1+\overline{\mbf Q}_M^1\cdot \overline{\mbf Q}_M^1 \rangle_h+ \langle\Delta t(\mbf V_M^1+\overline{\mbf V}_M^1), \overline{\mbf P}_M^1\cdot \overline{\mbf P}_M^1+\overline{\mbf Q}_M^1\cdot \overline{\mbf Q}_M^1 \rangle_h         \\
	&+\langle\Delta t(\mbf V_M^1+\overline{\mbf V}_M^1), \mbf P_M^{1} \cdot \mbf P_M^{1}+\mbf Q_M^{1} \cdot \mbf Q_M^{1}\rangle_h-\langle\mbf U_M^1+\overline{\mbf U}_M^1,\Delta t \mbf A \frac{\mbf Q_M^{1}+\overline{\mbf Q}_M^1}{2} \cdot (\mbf P_M^1 +\overline{\mbf P}_M^1)\rangle_h\\
	&+\langle\mbf U_M^1+\overline{\mbf U}_M^1,\Delta t \mbf A \frac{\mbf P_M^{1}+\overline{ \mbf P}_M^1}{2} \cdot (\mbf Q_M^1 +\overline{\mbf Q}_M^1)\rangle_h\\
	=& 2\langle \overline{\mbf U}_M^1, \overline{\mbf P}_M^1\cdot \overline{\mbf P}_M^1+\overline{\mbf Q}_M^1\cdot \overline{\mbf Q}_M^1 \rangle_h+\Delta t\langle \mbf V_M^1+\overline{\mbf V}_M^1, \mbf P_M^1\cdot \mbf P_M^1+ \mbf Q_M^1\cdot  \mbf Q_M^1 +\overline{\mbf P}_M^1\cdot \overline{\mbf P}_M^1+\overline{\mbf Q}_M^1\cdot \overline{\mbf Q}_M^1 \rangle_h\\
	& - \Big\langle \mbf U_M^1 +\overline{\mbf U}_M^1, \Delta t\mbf A \frac{\mbf Q_M^{1}+\overline{\mbf Q}_M^1}{2} \cdot (\mbf P_M^1 +\overline{\mbf P}_M^1)\Big\rangle_h
	+\Big\langle \mbf U_M^1 +\overline{\mbf U}_M^1, \Delta t\mbf A \frac{\mbf P_M^{1}+\overline{\mbf P}_M^1}{2} \cdot (\mbf Q_M^1 +\overline{\mbf Q}_M^1)\Big\rangle_h.
	\end{align*}
	Combining the above equations leads to
	\begin{equation}\label{eq5.14}
	\mathbb{E}[\mathcal{H}({\mbf P}_M^{1},{\mbf Q}_M^{1},{\mbf U}_M^{1},{\mbf V}_M^{1})  ]=\mathbb{E}[\mathcal{H}(\overline{\mbf P}_M^{1},\overline{\mbf Q}_M^{1},\overline{\mbf U}_M^{1},\overline{\mbf V}_M^{1})  ],
	\end{equation}
	which finishes the proof. 
\end{proof}

\begin{rem}
If we use the midpoint method, we obtain a fully-discrete scheme as follows
\begin{subequations}\label{MSP}
	\begin{align}
	\mbf P_M^{n} = &\mbf P_M^{n-1}-\Delta t\left( \mbf A \frac{\mbf Q_M^{n}+\mbf Q_M^{n-1}}{2}+\frac{1}{4}\left(\mbf U_M^{n}+\mbf U_M^{n-1}\right)\cdot\left(\mbf Q_M^{n}+\mbf Q_M^{n-1}\right)\right)+C_1\boldsymbol{\eta}_1\Delta B_1^n  \nonumber\\
	&+\frac{1}{2}\Delta t C_1\Delta B_0^n\left(
	\mbf A\boldsymbol{\eta}_1+\frac{\mbf U_M^n+\mbf U_M^{n-1}}{2}\cdot\boldsymbol{\eta}_1\right),\\
	\mbf Q_M^{n} = &\mbf Q_M^{n-1}+\Delta t\left( \mbf A \frac{\mbf P_M^{n}+\mbf P_M^{n-1}}{2}+\frac{1}{4}\left(\mbf U_M^{n}+\mbf U_M^{n-1}\right)\cdot\left(\mbf P_M^{n}+\mbf P_M^{n-1}\right)\right)-C_1\boldsymbol{\eta}_1\Delta B_0^n  \nonumber\\
	&+\frac{1}{2}\Delta t C_1\Delta B_1^n\left(
	\mbf A\boldsymbol{\eta}_1+\frac{\mbf U_M^n+\mbf U^{n-1}}{2}\cdot\boldsymbol{\eta}_1\right),\\
	\mbf V_M^{n} =&\mbf V_M^{n-1}+\frac{1}{2}\Delta t\left(  \mbf A \frac{\mbf U_M^{n}+\mbf U_M^{n-1}}{2}-\frac{\mbf U_M^{n}+\mbf U_M^{n-1}}{2}+\frac{1}{2}\left(\mbf P_M^{n-1} \cdot \mbf P_M^{n-1}+\mbf Q_M^{n-1} \cdot \mbf Q_M^{n-1}\right)\right)\nonumber\\
	&+\frac{1}{2}C_2\boldsymbol{\eta}_2\Delta B_2^n+\frac{1}{2}\Delta t\left(C_1\mbf P_M^{n-1}\cdot\boldsymbol{\eta}_1\Delta B_1^n-C_1\mbf Q_M^{n-1}\cdot\boldsymbol{\eta}_1\Delta B_0^n+\Delta tC_1^2\boldsymbol{\eta}_1\cdot\boldsymbol{\eta}_1\right), \\
	\mbf U_M^{n} =&\mbf U_M^{n-1}+\Delta t\left(\mbf V_M^{n}+\mbf V_M^{n-1}\right)+\frac{1}{2}\Delta t C_2\boldsymbol{\eta}_2\Delta B_2^n,
	\end{align}
\end{subequations}
which preserves both the averaged energy evolution law and averaged charge evolution law via similar arguments.
\end{rem}

 
When simulating various partial differential equations, sine pseudo-spectral methods play an important role due to their superior properties like high-order accuracy, good stability, and high efficiency (see \cite{Hesthaven,Thalhammer}). 
Now we adopt the sine pseudo-spectral method to approximate the stochastic KGS equation \eqref{SKGS} in space.  Concretely, we let $I_M$ be the trigonometric polynomial interpolation operator onto $\mathcal{S}_M :=\operatorname{span}\{\sin \left(\ell \mu (x-a)\right), $ $\ell=1,2, \ldots, M-1\}$,  with $ \mu=\frac{\pi}{b-a}$, i.e.,
$$
\left(I_M u\right)(x):=\sum_{\ell=1}^{M-1} \tilde{u}_{\ell} \sin \left({\ell}\mu(x-a)\right),~~
\tilde{u}_{\ell}:=\frac{2}{M} \sum_{i=1}^{M-1} u_i \sin \left({\ell}\mu\left(x_i-a\right)\right),
$$
where $u_i$ is interpreted as $u\left(x_i\right)$.  
Substituting $\tilde{u}_{\ell}$ into $I_M u$, we derive
\begin{equation}
\left(I_M u\right)(x)=\sum_{i=1}^{M-1} u_i X_i(x)
\end{equation}
with the interpolation basis function
\begin{equation}
X_i(x)=\frac{2}{M} \sum_{\ell=1}^{M-1} \sin \left({\ell}\mu(x_i-a)\right) \sin \left({\ell}\mu\left(x-a\right)\right).
\end{equation}
To obtain an approximation of the second-order derivative, we differentiate $X_i(x)$ twice, evaluate the resulting expressions of $X_i(x)$ at the collocation points $x_j$, and then derive the elements of  second-order differentiation operator $\mbf {\tilde{A}}$ as follows
\begin{equation}
\mbf {\tilde{A}}_{i, j}=\left\{\begin{aligned}
&(-1)^{i+j+1} \frac{\mu^2}{2}\left[\csc ^2\left(\frac{\mu}{2}(i-j) h\right)-\csc ^2\left(\frac{\mu}{2}(i+j) h\right)\right], & \text{if   } i \neq j, \\
&-\frac{\mu^2}{6}-\frac{M^2 \mu^2}{3}+\frac{\mu^2}{2} \csc ^2\left(i \mu h\right), & \text{if   } i=j.
\end{aligned}\right.
\end{equation}
Applying the sine pseudo-spectral method, we derive the semi-discrete scheme in matrix-vector form
\begin{equation}
\left\{
\begin{aligned}
&d\mbf Q_M(t) = \left(\mbf {\tilde{A}}\mbf P_M(t) +\mbf U_M(t)\cdot \mbf P_M(t)\right)dt - C_1\boldsymbol{\eta}_1dB_0(t), \\
&d\mbf P_M(t)= -\left(\mbf {\tilde{A}}\mbf Q_M(t) + \mbf U_M(t)\cdot \mbf Q_M(t)\right)dt + C_1\boldsymbol{\eta}_1dB_1(t), \\
&d\mbf V_M(t) = \frac{1}{2} \left(\mbf {\tilde{A}}\mbf U_M(t) - \mbf U_M(t) +\mbf P_M(t)\cdot \mbf P_M(t)+\mbf Q_M(t)\cdot \mbf Q_M(t)\right)dt + \frac{1}{2}C_2\boldsymbol{\eta}_2dB_2(t), \\
&d\mbf U_M(t) = 2 \mbf V_M(t) dt,
\end{aligned}
\right.
\end{equation}
Applying the modified techniques to time  discretization, we have the fully-discrete scheme
\begin{subequations}\label{CEFD-SPS}	
	\begin{align}
	\mbf P_M^{n} = &\mbf P_M^{n-1}-\Delta t\left( \mbf {\tilde{A}} \frac{\mbf Q_M^{n}+\mbf Q_M^{n-1}}{2}+\frac{1}{2}\mbf U_M^{n}\cdot\left(\mbf Q_M^{n}+\mbf Q_M^{n-1}\right)\right)+C_1\boldsymbol{\eta}_1\Delta B_1^n\nonumber\\
	&+\frac{1}{2}\Delta t C_1\Delta B_0^n\left(
	\mbf {\tilde{A}}\boldsymbol{\eta}_1+\mbf U_M^n\cdot\boldsymbol{\eta}_1\right),\\
	\mbf Q_M^{n} = &\mbf Q_M^{n-1}+\Delta t\left( \mbf {\tilde{A}} \frac{\mbf P_M^{n}+\mbf P_M^{n-1}}{2}+\frac{1}{2}\mbf U_M^{n}\cdot\left(\mbf P_M^{n}+\mbf P_M^{n-1}\right)\right)-C_1\boldsymbol{\eta}_1\Delta B_0^n\nonumber\\
	&+\frac{1}{2}\Delta t C_1\Delta B_1^n\left(
	\mbf{\tilde{A}}\boldsymbol{\eta}_1+\mbf U_M^n\cdot\boldsymbol{\eta}_1\right),\\	
	\mbf V_M^{n} =&\mbf V_M^{n-1}+\frac{1}{2}\Delta t\left(  \mbf {\tilde{A}} \frac{\mbf U_M^{n}+\mbf U_M^{n-1}}{2}-\frac{\mbf U_M^{n}+\mbf U_M^{n-1}}{2}+\mbf P_M^{n-1} \cdot \mbf P_M^{n-1}+\mbf Q_M^{n-1} \cdot \mbf Q_M^{n-1}\right)\nonumber\\
	&+\frac{1}{2}C_2\boldsymbol{\eta}_2\Delta B_2^n+\Delta t\left(C_1\mbf P_M^{n-1}\cdot\boldsymbol{\eta}_1\Delta B_1^n-C_1\mbf Q_M^{n-1}\cdot\boldsymbol{\eta}_1\Delta B_0^n+\Delta tC_1^2\boldsymbol{\eta}_1\cdot\boldsymbol{\eta}_1\right),\\
	\mbf U_M^{n} =&\mbf U_M^{n-1}+\Delta t\left(\mbf V_M^{n}+\mbf V_M^{n-1}\right)+\frac{1}{2}\Delta t C_2\boldsymbol{\eta}_2\Delta B_2^n,
	\end{align}
\end{subequations}
where $n\in\{1,\ldots,N\}.$  
Making use of similar procedures as in Theorems \ref{thm5.1} and \ref{thm5.2}, we have the following result. 

\begin{rem} 
	The fully-discrete scheme \eqref{CEFD-SPS} satisfies averaged charge evolution law
	\begin{equation*}
	\mathbb{E}[\mathcal{N}(\mbf P_M^n,\mbf Q_M^n)  ] = \mathbb{E}[\mathcal{N}(\mbf P_M^0,\mbf Q_M^0)]  + 2C_1^2\|\boldsymbol{\eta}_1\|_h^2t_n,
	\end{equation*}
	and averaged energy evolution law
	\begin{align*}
	&\mathbb{E}[ \overline{\mathcal{H}}(\mbf P_M^n,\mbf Q_M^n,\mbf U_M^n,\mbf V_M^n)  ] \\
	=&\mathbb{E}[ \overline{\mathcal{H}}(\mbf P_M^{0},\mbf Q_M^{0},\mbf U_M^{0},\mbf V_M^{0})] -C_2^2\mathcal{Q}_2t_n + 4C_1^2 \langle \boldsymbol{\eta}_1, \mbf {\tilde{A}}\boldsymbol{\eta}_1\rangle_h t_n+ 4C_1^2\sum_{i=0}^{n-1}\mathbb{E}[\langle  \mbf U_M^{i}, \boldsymbol{\eta}_1\cdot\boldsymbol{\eta}_1 \rangle_h]\Delta t,
	\end{align*}
	where $n\in\{0,1,\ldots,N\},$ 
	\begin{align*}
	&\overline{\mathcal{H}}(\mbf P_M^n,\mbf Q_M^n,\mbf U_M^n,\mbf V_M^n) \\
	=  &2\langle \mbf U_M^n, \mbf P_M^n\cdot \mbf P_M^n+ \mbf Q_M^n\cdot \mbf Q_M^n\rangle_h- \| \mbf U_M^n\|_h^2 - 4\| \mbf V_M^n\|_h^2\\
	&+ \langle \mbf U_M^n, \mbf {\tilde{A}}\mbf U_M^n\rangle_h+ 2  \langle \mbf P_M^n, \mbf {\tilde{A}}\mbf P_M^n\rangle_h +  2\langle \mbf Q_M^n, \mbf {\tilde{A}}\mbf Q_M^n\rangle_h.
	\end{align*}
\end{rem}

The finite element method, as a type of classic and mature numerical method, can deal with the irregular computational domain and have high flexibility (see \cite{MSF,LWX}).  
We now apply the finite element method to discretize \eqref{SKGS}.  First, we introduce some notations.
Let $\mathcal{T}_h$ be the uniform partition of $[a, b]$ with step size $h$ defined above, and denote $I_i = (x_i, x_{i+1}), i=0,1,2,\dots, M-1$. 
Set $\mathcal V_h$ to be the space of piecewise linear continuous functions with respect to $\mathcal{T}_h$ which vanish on the boundary of $[a, b]$. 
Multiplying $u_{xx}$ by $\zeta(x)$ and integrating by parts on each interval $I_i$ respectively, with $\zeta(x)$ being  functions in $\mathcal V_h$ for $i\in\{0,1,2,\dots, M-1\}$, we get
\begin{equation*}
\int_{x_i}^{x_{i+1}} u_{xx}(x)\zeta(x)(x)dx=-\int_{x_i}^{x_{i+1}} u_x(x) \zeta_x(x)dx.
\end{equation*}
Summing the above equations from $i=0$ to $M-1$, we obtain the following discrete bilinear form
\begin{equation}
\label{eq3.3}
B_h(u, \zeta) = -\sum_{i=0}^{M-1}\int_{x_i}^{x_{i+1}} u_x(x) \zeta_x(x)dx,
\end{equation}
which defines a discrete linear operator 
$A_h: \mathcal V_h\rightarrow \mathcal V_h$ as
\begin{equation*}
( A_h U_h, \zeta ) =B_h(U_h, \zeta) \quad \forall~ \zeta,U_h\in\mathcal V_h.
\end{equation*}
As a consequence, the finite element approximation of stochastic KGS equation \eqref{SKGS} can be regarded as: find $Q_h, P_h, V_h, U_h\in\mathcal V_h$ such that
\begin{equation}\label{FEM}
\left\{
\begin{aligned}
&dQ _h= \left(A_hP_h+\mathcal{P}_h U_hP_h\right)dt - C_1\mathcal{P}_h\eta_1dB_0(t), \\
&dP_h = -\left(A_hQ_h + \mathcal{P}_hU_hQ_h\right)dt +C_1\mathcal{P}_h\eta_1dB_1(t), \\
&dV_h = \frac{1}{2}\left( A_hU_h -  U_h + \mathcal{P}_h(P_h^2+Q_h^2) \right)dt + \frac{1}{2}C_2\mathcal{P}_h\eta_2dB_2(t), \\
&dU_h = 2 V_h dt, 
\end{aligned}
\right.
\end{equation}
where $\mathcal{P}_h: L^2([a, b])\rightarrow\mathcal V_h$ is the projection defined by $(\mathcal{P}_h\mu, \phi)=(\mu,\phi), ~\forall \phi \in \mathcal V_h$.
Further, we have 
\begin{align}
\label{CEFD-FEM}
P_h^n =& {P}_h^{n-1}-\Delta t\left( A_h \frac{Q_h^n+{Q}_h^{n-1}}{2}+\frac{1}{2}\mathcal{P}_hU_h^n\left(Q_h^n+{Q}_h^{n-1}\right)\right)+C_1\mathcal{P}_h\eta_1\Delta B_1^n\nonumber\\
&+\frac{1}{2}\Delta tC_1\left( A_h \mathcal{P}_h\eta_1\Delta B_0^n+\mathcal{P}_hU_h^n\mathcal{P}_h\eta_1\Delta B_0^n\right),\nonumber\\
Q_h^n =&{Q}_h^{n-1}+\Delta t\left(A_h \frac{P_h^n+{P}_h^{n-1}}{2}+\frac{1}{2}\mathcal{P}_h U_h^n\left(P_h^n+{P}_h^{n-1}\right)\right)- C_1\mathcal{P}_h\eta_1\Delta B_0^n\nonumber\\
&+\frac{1}{2}\Delta tC_1\left( A_h \mathcal{P}_h\eta_1\Delta B_1^n+\mathcal{P}_hU_h^n\mathcal{P}_h\eta_1\Delta B_1^n\right), \\
V_h^n =&{V}_h^{n-1}+\Delta t\left(\frac{1}{2} A_h \frac{U_h^n+{U}_h^{n-1}}{2}-\frac{1}{2}\frac{U_h^n+{U}_h^{n-1}}{2}+\frac{1}{2}\mathcal{P}_h\left(({P}_h^{n-1})^2 +({Q}_h^{n-1})^2\right)\right)\nonumber\\
&+\frac{1}{2}C_2\mathcal{P}_h\eta_2 \Delta B_2^n+\Delta t C_1\mathcal{P}_h\left({P}_h^{n-1}\mathcal{P}_h\eta_1\Delta B_1^n-{Q}_h^{n-1}\mathcal{P}_h\eta_1\Delta B_0^n+C_1\Delta t(\mathcal{P}_h\eta_1)^2\right),\nonumber\\
U_h^n =&{U}_h^{n-1}+\Delta t\left(V_h^n+{V}_h^{n-1}\right)+ \frac{1}{2}\Delta tC_2\mathcal{P}_h\eta_2 \Delta B_2^n.\nonumber
\end{align}
\begin{rem}
	The fully-discrete scheme \eqref{CEFD-FEM} preserves the averaged charge evolution law
	\begin{equation*}
	\mathbb{E}[\mathcal{\check N}(P_h^n,Q_h^n)] = \mathbb{E}[\mathcal{\check N}(P_h^{0},Q_h^{0})]  + 2C_1^2( \mathcal{P}_h\eta_1, \mathcal{P}_h\eta_1) t_n
	\end{equation*}
	with $\mathcal{\check N}(P_h^n,Q_h^n)= (P_h^n, P_h^n ) + ( Q_h^n, Q_h^n )$, 
	and averaged energy evolution law
	\begin{align*}
	&\mathbb{E}[\mathcal{\check H}( P_h^n, Q_h^n,U_h^n, V_h^n)  ]
	=\mathbb{E}[\mathcal{\check H}( P_h^{0}, Q_h^{0}, U_h^{0}, V_h^{0})] -C_2^2\mathcal{\check Q}_2t_n + 4C_1^2 \mathcal{\check Q}_3t_n+ 4C_1^2\sum_{i=0}^{n-1}\mathbb{E}[( U_h^{i}, (\mathcal{P}_h \eta_1)^2 )]\Delta t,
	\end{align*}
	where $n\in\{0,1,\ldots,N\},$ 
	\begin{align*}
	\mathcal{\check H}(P_h^n,Q_h^n, U_h^n, V_h^n) =  &2( U_h^n, \mathcal{P}_h (P_h^n)^2+ \mathcal{P}_h(Q_h^n)^2 ) - ( U_h^n, U_h^n) -4(V_h^n, V_h^n) \\
	&+ (U_h^n, A_h U_h^n) + 2 (P_h^n, A_hP_h^n) +  2(Q_h^n, A_hQ_h^n),
	\end{align*}
	and $\mathcal{\check Q}_2 =(\mathcal{P}_h\eta_2, \mathcal{P}_h\eta_2), \mathcal{\check Q}_3 = (\mathcal{P}_h\eta_1, A_h\mathcal{P}_h\eta_1)$.	
\end{rem}


\section{Symplectic and multi-symplectic method}

For a stochastic Hamiltonian system, symplectic methods are shown to be superior to nonsymplectic ones especially in long time computation, owing to their preservation of the symplecticity of the underlying continuous differential equation system \cite{Anton, Song}. 
In this section, we present symplectic and multi-symplectic methods for \eqref{KGS}.

Runge--Kutta methods, as a class of efficient derivative-free numerical methods, are important tools for the treatment of stochastic Hamiltonian systems. 
Beneath the complexity and variety, all Runge--Kutta methods have a common form that can be summarized by a matrix and two vectors. 
In detail, 
for $s$-stage Runge–Kutta methods with $s\ge1$, the corresponding Butcher chart reads
 \begin{equation*}
\label{SRK}
\begin{array}{c|cccc}
c_{1} &  a_{11} &\cdots &a_{1s}\\
\vdots&  \vdots &\ddots &\vdots\\
c_{s} & a_{s1} &\cdots& a_{ss}\\
\hline
& b_{1} &\cdots  &b_{s}
\end{array}.
\end{equation*}
By exploiting the symplectic Runge-Kutta method, we get a temporal discretization for \eqref{KGS} as follows
\begin{subequations}\label{SRK}
	\begin{align*}
	&Q^{n,m} = Q^{n-1}+  \sum_{l=1}^{s}a_{ml}\left( \frac{1}{2} P_{xx}^{n,l}\Delta t +U^{n,l}P^{n,l}\Delta t -C_1\eta_1\Delta B_0^{n}\right), \\
	&Q^n = Q^{n-1} +  \sum_{m=1}^{s}b_{m}\left( \frac{1}{2} P_{xx}^{n,m}\Delta t +U^{n,m}P^{n,m}\Delta t -C_1\eta_1\Delta B_0^{n}\right), \\
	& P^{n,m}  = P^{n-1} -  \sum_{l=1}^{s}a_{ml}\left( \frac{1}{2} Q_{xx}^{n,l}  \Delta t+U^{n,l}Q^{n,l}  \Delta t - C_1\eta_1\Delta B_1^{n}\right), \\
	&P^n = P^{n-1} -  \sum_{m=1}^{s}b_{m}\left( \frac{1}{2} Q_{xx}^{n,m} \Delta t+U^{n,m} Q^{n,m} \Delta t - C_1\eta_1\Delta B_1^{n}\right), \\
	&V^{n,m} = V^{n-1} + \frac{1}{2} \sum_{l=1}^{s}a_{ml}\left( U_{xx}^{n,l}\Delta t - U^{n,l}\Delta t  +(P^{n,l})^2\Delta t +(Q^{n,l})^2\Delta t +C_2\eta_2\Delta B_2^{n}\right), \\
	&V^n= V^{n-1} + \frac{1}{2} \sum_{m=1}^{s}b_{m}\left( U_{xx}^{n,m}\Delta t- U^{n,m}\Delta t  +(P^{n,m})^2\Delta t  +(Q^{n,m})^2\Delta t +C_2\eta_2\Delta B_2^{n}\right), \\
	&U^{n,m} = U^{n-1} +  2\Delta t \sum_{l=1}^{s}a_{ml}V^{n,l}, ~~~U^n = U^{n-1}+ 2\Delta t \sum_{m=1}^{s}b_{m}V^{n,m},
	\end{align*}
\end{subequations}
where $n\in\{1,\ldots,N\},$  
\begin{equation}
\label{SC}
a_{ij}b_j+a_{ji}b_i=b_ib_j,\qquad i,j=1,\ldots,s.
\end{equation}
Here, $Q^n,P^n,V^n,U^n$ are approximations of $q(t_n),$ $p(t_n),$ $v(t_n),$ $u(t_n),$  respectively, $Q^{n,m},$ $P^{n,m},$ $V^{n,m},$ $U^{n,m}$ stand for the approximation of $q(t_{n,m}),$ $p(t_{n,m}),$  $v(t_{n,m}),$ $u(t_{n,m})$ which are the $m$th middle-value for $t_{n,m}\in (t_{n-1},t_n)$, 
and $Q_{xx}^{n,m},P_{xx}^{n,m},U_{xx}^{n,m}$  represent 
$\frac {\partial^2q}{\partial x^2}(t_{n,m})$, $\frac {\partial^2p}{\partial x^2}(t_{n,m}),$ $\frac {\partial^2u}{\partial x^2}(t_{n,m}),$ respectively.

\begin{ex}
	Let $s=2$, $a_{11}=\frac 14,$ $a_{21}=\frac{1}{4}+\frac{\sqrt{3}}{6}+\alpha$, $a_{12}=\frac{1}{4}-\frac{\sqrt{3}}{6}-\alpha,$ $a_{22}=\frac 14$, $b_1=b_2=\frac 12$ with $\alpha\in\mathbb R$. 
	Exploiting the parametric Runge--Kutta method to \eqref{KGS}, we obtain 
	\begin{subequations}
		\begin{align*}	
		Q^{n,1} = &Q^{n-1}+  \frac{1}{4}\left( \frac{1}{2} P_{xx}^{n,1}\Delta t +U^{n,1}P^{n,1}\Delta t -C_1\eta_1\Delta B_0^{n}\right)\\
		&+  \left(\frac{1}{4}-\frac{\sqrt{3}}{6}-\alpha\right)\left( \frac{1}{2} P_{xx}^{n,2}\Delta t +U^{n,2}P^{n,2}\Delta t -C_1\eta_1\Delta B_0^{n}\right),\\
		Q^{n,2} = &Q^{n-1}+   \left(\frac{1}{4}+\frac{\sqrt{3}}{6}+\alpha\right)\left( \frac{1}{2} P_{xx}^{n,1}\Delta t +U^{n,1}P^{n,1}\Delta t -C_1\eta_1\Delta B_0^{n}\right)\\
		&+ \frac{1}{4}\left( \frac{1}{2} P_{xx}^{n,2}\Delta t +U^{n,2}P^{n,2}\Delta t -C_1\eta_1\Delta B_0^{n}\right),\\	
		Q^n = &Q^{n-1} +  \frac{1}{2}\left( \frac{1}{2} P_{xx}^{n,1}\Delta t +U^{n,1}P^{n,1}\Delta t -C_1\eta_1\Delta B_0^{n}\right)+ \\
		& \frac{1}{2}\left( \frac{1}{2} P_{xx}^{n,2}\Delta t +U^{n,2}P^{n,2}\Delta t -C_1\eta_1\Delta B_0^{n}\right),\\	
		P^{n,1}  = &P^{n-1} -  \frac{1}{4}\left( \frac{1}{2} Q_{xx}^{n,1}  \Delta t+U^{n,1}Q^{n,1}  \Delta t- C_1\eta_1\Delta B_1^{n}\right),\\
		&- \left(\frac{1}{4}-\frac{\sqrt{3}}{6}-\alpha\right)\left( \frac{1}{2} Q_{xx}^{n,2}  \Delta t+U^{n,2}Q^{n,2}  \Delta t- C_1\eta_1\Delta B_1^{n}\right)\\
		P^{n,2}  = &P^{n-1} -  \left(\frac{1}{4}+\frac{\sqrt{3}}{6}+\alpha\right)\left( \frac{1}{2} Q_{xx}^{n,1}  \Delta t+U^{n,1}Q^{n,1}  \Delta t- C_1\eta_1\Delta B_1^{n}\right), \nonumber\\
		&- \frac{1}{4}\left( \frac{1}{2} Q_{xx}^{n,2}  \Delta t+U^{n,2}Q^{n,2}  \Delta t- C_1\eta_1\Delta B_1^{n}\right)\\
		P^n =&P^{n-1} -  \frac{1}{2}\left( \frac{1}{2} Q_{xx}^{n,1} \Delta t+U^{n,1} Q^{n,1} \Delta t -C_1\eta_1\Delta B_1^{n}\right)\\
		&-\frac{1}{2}\left( \frac{1}{2} Q_{xx}^{n,2} \Delta t+U^{n,2} Q^{n,2} \Delta t -C_1\eta_1\Delta B_1^{n}\right), \\
		V^{n,1} =& V^{n-1} +  \frac{1}{8}\left( U_{xx}^{n,1}\Delta t - U^{n,1}\Delta t  +(P^{n,1})^2\Delta t +(Q^{n,1})^2\Delta t +C_2\eta_2\Delta B_2^{n}\right)\\	
		&+  \left(\frac{1}{8}-\frac{\sqrt{3}}{12}-\frac \alpha 2\right)\left( U_{xx}^{n,2}\Delta t - U^{n,2}\Delta t  +(P^{n,2})^2\Delta t +(Q^{n,2})^2\Delta t +C_2\eta_2\Delta B_2^{n}\right)\\
		V^{n,2} =& V^{n-1} + \left(\frac{1}{8}+\frac{\sqrt{3}}{12}+\frac \alpha 2\right) \left( U_{xx}^{n,1}\Delta t - U^{n,1}\Delta t  +(P^{n,1})^2\Delta t +(Q^{n,1})^2\Delta t +C_2\eta_2\Delta B_2^{n}\right)\\	
		&+ \frac{1}{8}\left( U_{xx}^{n,2}\Delta t - U^{n,2}\Delta t  +(P^{n,2})^2\Delta t +(Q^{n,2})^2\Delta t +C_2\eta_2\Delta B_2^{n}\right),\\
		V^n=&V^{n-1} + \frac{1}{4} \left( U_{xx}^{n,1}\Delta t- U^{n,1}\Delta t  +(P^{n,1})^2\Delta t  +(Q^{n,1})^2\Delta t +C_2\eta_2\Delta B_2^{n}\right)\\
		&+\frac{1}{4} \left( U_{xx}^{n,2}\Delta t- U^{n,2}\Delta t  +(P^{n,2})^2\Delta t  +(Q^{n,2})^2\Delta t +C_2\eta_2\Delta B_2^{n}\right),\\
		U^{n,1} = &U^{n-1} +  \frac{1}{2}  \Delta t V^{n,1}+  \left(\frac{1}{2} -\frac{\sqrt{3}}{3} -2\alpha\right)\Delta t V^{n,2}, \\
		U^{n,1} =& U^{n-1} +   \left(\frac{1}{2} +\frac{\sqrt{3}}{3} +2\alpha\right) \Delta t V^{n,1}+ \frac{1}{2} \Delta t V^{n,2}, \\
		U^n =&U^{n-1}+ \Delta t V^{n,1}+\Delta t V^{n,2}.
		\end{align*}
	\end{subequations}
	When $\alpha=0$, the corresponding parametric Runge-Kutta method becomes the traditional Legendre-Gauss collocation method. 
	
	When $s>2,$ the family of parametric Runge--Kutta methods can be defined by the following tableau
	\begin{equation}
	\label{RK}
	\begin{array}{c|c}
	c_{1} &   \\
	\vdots&   \mathcal{A}(\alpha)=l X_s (\alpha)l^{-1} \\
	c_{s} & \\
	\hline
	& b_{1} \dots  b_{s}
	\end{array},
	\end{equation}
	where 
	\begin{align*}
	&l=\left[\begin{array}{cccc}
	l_1(c_1) & l_2(c_1) &\dots&l_s(c_1)\\
	l_1(c_2)  & l_2(c_2) &\dots&l_s(c_2)\\
	\vdots & \vdots &&\vdots\\
	l_1(c_s)  & l_2(c_s) &\dots&l_s(c_s)\\
	\end{array}\right]_{s\times s}, 
	\quad X_s (\alpha)=\left[\begin{array}{cccc}
	\frac{1}{2} & -(\xi_1+\alpha_1)&&\\
	\xi_1+\alpha_1 & 0 &\ddots&\\
	& \ddots &\ddots& -(\xi_{s-1}+\alpha_{s-1})\\
	&  &\xi_{s-1}+\alpha_{s-1}&0\\
	\end{array}\right]
	\end{align*} 
	with $\alpha_1, \cdots, \alpha_{s-1}\in\mathbb R$, $\xi_i = \frac{1}{2\sqrt{(2i+1)(2i-1)}},i= 1,\dots,s-1$, and $l_i(\tau)$ being the Legendre polynomials of 
	degree $i-1$ shifted and normalized in the interval $[0,1]$ for $i=1,\dots, s$ satisfying
	\begin{equation*}
	\int_0^1l_i(\tau)l_j(\tau)d\tau =\delta_{ij}, \quad i, j=1,\dots, s.
	\end{equation*}
	In this case, $c_1\leq c_2\leq \dots \leq c_s$ and $b_1, \dots, b_s$ are the abscissae and the weights of the Gauss--Legendre quadrature formula in the interval $[0,1]$, respectively. 
	For any value of $\alpha_1, \cdots, \alpha_{s-1}$, the corresponding parametric Runge--Kutta method defined by \eqref{RK}
	is symplectic, since
	\begin{equation*}
	\label{SC}
	B \mathcal{A}(\alpha) +  \mathcal{A}(\alpha)^\top B = bb^\top
	\end{equation*}
	holds, where $B= diag \{b_1,\dots, b_s\}$, $b=[b_1,\dots, b_s]^\top$. 
	Introducing parameters into Gauss collocation formulae leads to the numerical method which is symplectic and much greater degree of flexibility. 
\end{ex}

In the following, we prove that the semi-discrete scheme \eqref{SRK}  preserves the discrete symplectic conservation law almost surely.
\begin{thm}
The temporal discretization \eqref{SRK} admits the discrete symplectic conservation law, i.e., 
\begin{align*}
\mathrm{~d} Q^1 \wedge \mathrm{d} P^1 + \mathrm{d} V^1 \wedge \mathrm{d} U^1
=\mathrm{~d} Q^0 \wedge \mathrm{d} P^0 + \mathrm{d} V^0 \wedge \mathrm{d} U^0,  \quad a.s. .
\end{align*}
\end{thm}
\begin{proof}
By utilizing \eqref{SRK},  we obtain
\begin{equation}\begin{aligned} \label{eq4.1}
&\mathrm{d}Q^{1} \wedge \mathrm{d}P^{1} - \mathrm{d}Q^{0} \wedge \mathrm{d}P^{0}\\
=&  -\mathrm{d} Q^0 \wedge \Delta t \sum_{m=1}^{s}b_{m}\left( \frac{1}{2}\mathrm{d} Q_{xx}^{1,m}+Q^{1,m}\mathrm{d}U^{1,m} + U^{1,m}\mathrm{d}Q^{1,m} \right)\\
&+  \Delta t\sum_{m=1}^{s}b_{m}\left( \frac{1}{2} \mathrm{d}P_{xx}^{1,m} +P^{1,m}\mathrm{d}U^{1,m}  + U^{1,m}\mathrm{d}P^{1,m}\right) \wedge \mathrm{d} P^0 \\
&-  \Delta t^2 \sum_{k,m=1}^{s}b_{m}\left( \frac{1}{2} \mathrm{d}P_{xx}^{1,m}+P^{1,m}\mathrm{d}U^{1,m}  + U^{1,m}\mathrm{d}P^{1,m} \right) \wedge b_{k}\left( \frac{1}{2} \mathrm{d}Q_{xx}^{1,k}+Q^{1,k}\mathrm{d}U^{1,k} + U^{1,k}\mathrm{d}Q^{1,k} \right).
\end{aligned}
\end{equation}
From
\begin{align*}
&\mathrm{d}Q^0 = \mathrm{d}Q^{1,m}  - \Delta t \sum_{l=1}^{s}a_{ml}\left( \frac{1}{2} \mathrm{d}P_{xx}^{1,l} +P^{1,l} \mathrm{d}U^{1,l}  + U^{1,l} \mathrm{d}P^{1,l}  \right), \\
& \mathrm{d}P^0 =\mathrm{d}P^{1,m}  +  \Delta t \sum_{l=1}^{s}a_{ml}\left( \frac{1}{2} \mathrm{d}Q_{xx}^{1,l}  +Q^{1,l} \mathrm{d}U^{1,l}   + U^{1,l} \mathrm{d}Q^{1,l} \right)
\end{align*}
and symplectic condition \eqref{SC} it follows that
\begin{align}\label{eq4.2}
\mathrm{d}Q^{1} \wedge \mathrm{d}P^{1} =& \mathrm{d}Q^{0} \wedge \mathrm{d}P^{0}
- \mathrm{d}Q^{1,m} \wedge \Delta t\sum_{m=1}^{s}b_{m}\left( \frac{1}{2}\mathrm d Q_{xx}^{1,m}+Q^{1,m}\mathrm{d}U^{1,m}\right)\nonumber\\
&+ \Delta t\sum_{m=1}^{s}b_{m}\left(\frac{1}{2}\mathrm d P_{xx}^{1,m}+P^{1,m} \mathrm{d}U^{1,m}\right) \wedge \mathrm{d}P^{1,m}.
\end{align}
Similarly,
\begin{equation*}
\begin{aligned}
&\mathrm{d}V^{1} \wedge \mathrm{d}U^{1}
\\
=& \mathrm{d}V^{0} \wedge \mathrm{d}U^{0}
+\mathrm{d} V^0 \wedge 2\Delta t \sum_{m=1}^{s}b_{m}\mathrm{d}V^{1,m}\\
&+  \frac{1}{2} \Delta t\sum_{m=1}^{s}b_{m}\left( \mathrm{d}AU^{1,m}- \mathrm{d}U^{1,m}  +2P^{1,m}\mathrm{d}P^{1,m}  +2Q^{1,m}\mathrm{d}Q^{1,m} \right) \wedge \mathrm{d} U^0 \\
&+  \Delta t\sum_{m=1}^{s}b_{m}\left( \mathrm{d}AU^{1,m}- \mathrm{d}U^{1,m}  +2P^{1,m}\mathrm{d}P^{1,m} +2Q^{1,m}\mathrm{d}Q^{1,m}\right)  \wedge \Delta t \sum_{k=1}^{s}b_{k}\mathrm{d}V^{1,k}.
\end{aligned}\end{equation*}
Further, based on  symplectic condition \eqref{SC}, we obtain
\begin{equation}\label{eq4.3}
\mathrm{d}V^{1} \wedge \mathrm{d}U^{1} = \mathrm{d}V^{0} \wedge \mathrm{d}U^{0}
+  \Delta t\sum_{m=1}^{s}b_{m}\big(\frac{1}{2}\mathrm d U_{xx}^{1.m}+P^{1.m}\mathrm{d}P^{1.m} +Q^{1.m} \mathrm{d}Q^{1.m}\big)\wedge\mathrm{d}U^{1.m}.
\end{equation}
We combine \eqref{eq4.2} and \eqref{eq4.3} to get 
\begin{align*}
&\mathrm{~d} Q^1 \wedge \mathrm{d} P^1 + \mathrm{d} V^1 \wedge \mathrm{d} U^1-(\mathrm{~d} Q^0 \wedge \mathrm{d} P^0 + \mathrm{d} V^0 \wedge \mathrm{d} U^0)\\
=&\frac{1}{2}\Delta t \sum_{m=1}^{s}b_m\big(\mathrm dQ_{xx}^{1,m}\wedge \mathrm d Q^{1,m}+\mathrm dP_{xx}^{1,m}\wedge \mathrm d P^{1,m}+\mathrm dU_{xx}^{1,m}\wedge \mathrm d U^{1,m}\big).
\end{align*}
Combining the property of $\partial_{xx},$, we have $\mathrm dQ_{xx}^{1,m}\wedge \mathrm d Q^{1,m}+\mathrm dP_{xx}^{1,m}\wedge \mathrm d P^{1,m}+\mathrm dU_{xx}^{1,m}\wedge \mathrm d U^{1,m}=0,$
which completes the proof.
\end{proof}

Applying the finite difference to  \eqref{multi-symplectic} yields the following semi-discretization
\begin{align}\label{MSSD}
\begin{cases}
dP_i(t)= -\left(\delta_xG_i(t) + U_i(t)\cdot Q_i(t)\right)dt + C_1\eta_1(x_i)dB_1(t), \\
dQ_i(t) = \left( \delta_xF_i(t) +U_i(t)\cdot P_i(t)\right)dt - C_1\eta_1(x_i)dB_0(t), \\
\delta_xP_i(t) = F_i(t),\\
\delta_xQ_i(t) = G_i(t),\\
dU_i(t) =R_i(t)dt,\\
dR_i(t) = \left(\delta_xO_i(t) - U_i(t) +P_i(t)\cdot P_i(t)+Q_i(t)\cdot Q_i(t)\right)dt + C_2\eta_2(x_i)dB_2(t), \\
\delta_x U_i(t) = O_i(t),
\end{cases}
\end{align}
where $i=1,2,\dots,M-1.$ Here, $\delta_x$ is the numerical approximation of one-order spatial derivative, for instance,  $\displaystyle\delta_xG_i(t)=\frac{G_{i+1}(t)-G_{i}(t)}{h},$ 
$\displaystyle\delta_xF_i(t)=\frac{F_{i+1}(t)-F_{i}(t)}{h},$ 
$\displaystyle\delta_xO_i(t)=\frac{O_{i+1}(t)-O_{i}(t)}{h}.$ 
Then utilizing the symplectic Runge--Kutta method to \eqref{MSSD} leads to
\begin{subequations}\label{MSFD}
	\begin{align}
	\label{MSFD-1}
	& P_i^{n,m} = P_i^{n-1} -  \sum_{l=1}^{s}a_{ml}\left(  \delta_xG_i^{n,l}\Delta t  +U_i^nQ_i^{n,l}\Delta t  + C_1\eta_1(x_i)\Delta B_1^n)\right), \\
	\label{MSFD-2}
	&P_i^n = P_i^{n-1}- \sum_{m=1}^{s}b_{m}\left(  \delta_xG_i^{n,m}\Delta t +U_i^{n,m}Q_i^{n,m} \Delta t + C_1\eta_1(x_i)\Delta B_1^n\right), \\
	\label{MSFD-3}
	&Q_i^{n,m} = Q_i^{n-1}+  \sum_{l=1}^{s}a_{ml}\left(  \delta_xF_i^{n,l} \Delta t +U_i^{n,l} P_i^{n,l} \Delta t -C_1\eta_1(x_i)\Delta B_0^n\right), \\
	\label{MSFD-4}
	&Q_i^n = Q_i^{n-1} +  \sum_{m=1}^{s}b_{m}\left(  \delta_xF_i^{n,m}\Delta t +U_i^{n,m}P_i^{n,m}\Delta t -C_1\eta_1(x_i)\Delta B_0^n\right), \\
	\label{MSFD-5}
	&\delta_xP_i^n = F_i^n, ~ \delta_xQ_i^n = G_i^n, \\
	\label{MSFD-6}
	&U_i^{n,m} = U_i^{n-1} + \Delta t \sum_{l=1}^{s}a_{ml}R_i^{n,l}, ~~~U_i^{n}= U_i^{n-1} + \Delta t \sum_{m=1}^{s}b_{m}R_i^{n,m},\\
	\label{MSFD-7}
	&R_i^{n,m} = R_i^0 +  \sum_{l=1}^{s}a_{ml}\left(\delta_xO_i^{n,l}\Delta t - U_i^{n,l}\Delta t  +(P_i^{n,l})^2\Delta t +(Q_i^{n,l})^2\Delta t +C_2\eta_2(x_i)\Delta B_2^n\right), \\
	\label{MSFD-8}
	&R_i^n = R_i^0 +  \sum_{m=1}^{s}b_{m}\left( \delta_xO_i^{n,m}\Delta t- U_i^{n,m}\Delta t  +(P_i^{n,m})^2\Delta t  +(Q_i^{n,m})^2\Delta t +C_2\eta_2(x_i)\Delta B_2^n\right), \\
	\label{MSFD-9}
	&\delta_xU_i^n = O_i^n,
	\end{align}
\end{subequations}
where $i=1,\ldots,M-1,$ $m=1,\ldots,s,$ and $n=1,\ldots,N.$ 

\begin{thm} 
The fully-discrete scheme \eqref{MSFD} admits the following discrete multi-symplectic conservation law 
\begin{equation}\begin{aligned}
&\frac{1}{\Delta t}(2\mathrm{d} Q_i^1 \wedge \mathrm{d} P_i^1+ \mathrm{d} R_i^1 \wedge \mathrm{d} U_i^1 -2\mathrm{d} Q_i^0 \wedge \mathrm{d} P_i^0- \mathrm{d} R_i^0 \wedge \mathrm{d} U_i^0 )=0,\quad{a.s.}.
\end{aligned}
\end{equation}
\end{thm}
\begin{proof}
According to the discrete variational equations of \eqref{MSFD-2} and \eqref{MSFD-4}, we obtain
\begin{align*}
&\mathrm{d}Q_{i}^{1} \wedge \mathrm{d}P_{i}^{1} - \mathrm{d}Q_{i}^{0} \wedge \mathrm{d}P_{i}^{0}\\
=&  -\mathrm{d} Q_i^0 \wedge \Delta t \sum_{m=1}^{s}b_{m}\left( \mathrm{d} \delta_xG_i^m+Q_i^m\mathrm{d}U_i^m + U_i^m\mathrm{d}Q_i^m \right)\\
&+\Delta t\sum_{m=1}^{s}b_{m}\left( \mathrm{d}\delta_xF_i^m +P_i^m\mathrm{d}U_i^m  + U_i^m\mathrm{d}P_i^m \right) \wedge \mathrm{d} P_i^0 \\
&-\Delta t \sum_{m=1}^{s}b_{m}\left(  \mathrm{d}\delta_xF_i^m+P_i^m\mathrm{d}U_i^m  + U_i^m\mathrm{d}P_i^m \right) \wedge \Delta t \sum_{m=1}^{s}b_{m}\left( \mathrm{d}\delta_xG_i^m+Q_i^m\mathrm{d}U_i^m + U_i^m\mathrm{d}Q_i^m \right).
\end{align*}
Applying \eqref{MSFD-1}, \eqref{MSFD-3} and the symplectic condition \eqref{SC} yields
\begin{equation}\begin{aligned}\label{eq4.4}
\mathrm{d}Q_{i}^{1} \wedge \mathrm{d}P_{i}^{1} =& \mathrm{d}Q_{i}^{0} \wedge \mathrm{d}P_{i}^{0}  -\mathrm{d} Q_i^m \wedge \Delta t \sum_{m=1}^{s}b_{m}\mathrm{d} \delta_xG_i^m 
-\mathrm{d} Q_i^m \wedge \Delta t \sum_{m=1}^{s}b_{m}Q_i^m \mathrm{d} U_i^m \\
&+ \Delta t\sum_{m=1}^{s}b_{m}\mathrm{d}\delta_xF_i^m \wedge \mathrm{d} P_i^m 
+ \Delta t\sum_{m=1}^{s}b_{m}P_i^m \mathrm{d}U_i^m\wedge \mathrm{d} P_i^m.
\end{aligned}\end{equation}
Analogously, we derive the following equation according to  \eqref{MSFD-6}, \eqref{MSFD-7}, \eqref{MSFD-8}
\begin{equation}\begin{aligned}\label{eq4.5}
\mathrm{d}R_{i}^{1} \wedge \mathrm{d}U_{i}^{1} =& \mathrm{d}R_{i}^{0} \wedge \mathrm{d}U_{i}^{0} 
+ \Delta t\sum_{m=1}^{s}b_{m}\mathrm{d}\delta_xO_i^m \wedge \mathrm{d} U_i^m 
+ 2\Delta t\sum_{m=1}^{s}b_{m}P_i^m \mathrm{d}P_i^m\wedge \mathrm{d} U_i^m\\
&+2\Delta t\sum_{m=1}^{s}b_{m}Q_i^m \mathrm{d}Q_i^m\wedge \mathrm{d} U_i^m.
\end{aligned}\end{equation}
Combining  \eqref{eq4.4} and \eqref{eq4.5}, we deduce
\begin{equation*}\begin{aligned}
&\frac{1}{\Delta t}(2\mathrm{d} Q_i^1 \wedge \mathrm{d} P_i^1+ \mathrm{d} R_i^1 \wedge \mathrm{d} U_i^1 -2\mathrm{d} Q_i^0 \wedge \mathrm{d} P_i^0- \mathrm{d} R_i^0 \wedge \mathrm{d} U_i^0 ) \\
=&\sum_{m=1}^{s}b_{m}( \mathrm{d}\delta_x F_i^m \wedge \mathrm{d} P_i^m + \mathrm{d}\delta_xG_i^m \wedge \mathrm{d} Q_i^m + \mathrm{d} \delta_x O_i^m \wedge \mathrm{d} U_i^m)dt=0,
\end{aligned} 
\end{equation*}
where implies the proof. 
\end{proof}

\section{Numerical experiments}
This section presents various numerical experiments to verify the properties of the proposed fully-discrete schemes for the 1-dimensional stochastic Klein-Gordon-Schr\"odinger equation with the homogeneous Dirichlet boundary condition. 
In all numerical experiments, the expectation is calculated by taking average over 2000 realizations.

\begin{figure}[h]
\centering
\includegraphics[height=4cm,width=4.5cm]{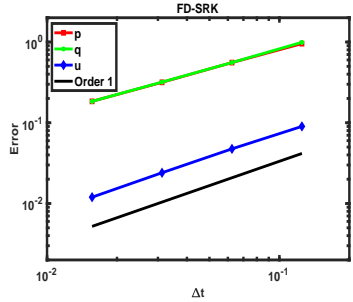}
\caption{Mean-square errors of FD-SRK in time with T=1.} 
\label{fig0}
\end{figure}

\begin{figure}[h]
	\centering
	\subfigure{
		\begin{minipage}{15cm}
			\centering
			\includegraphics[height=4cm,width=4.5cm]{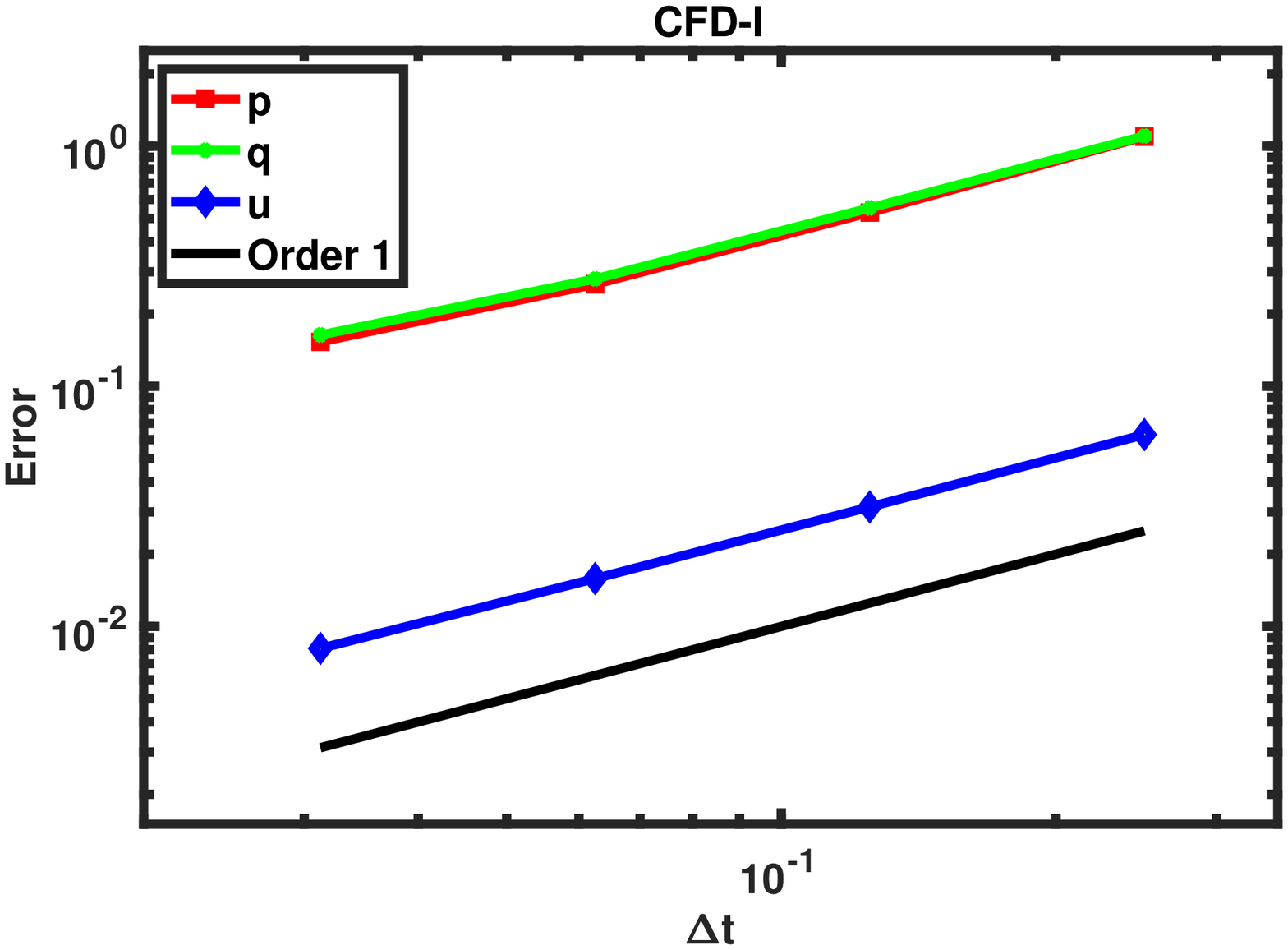}
			\includegraphics[height=4cm,width=4.5cm]{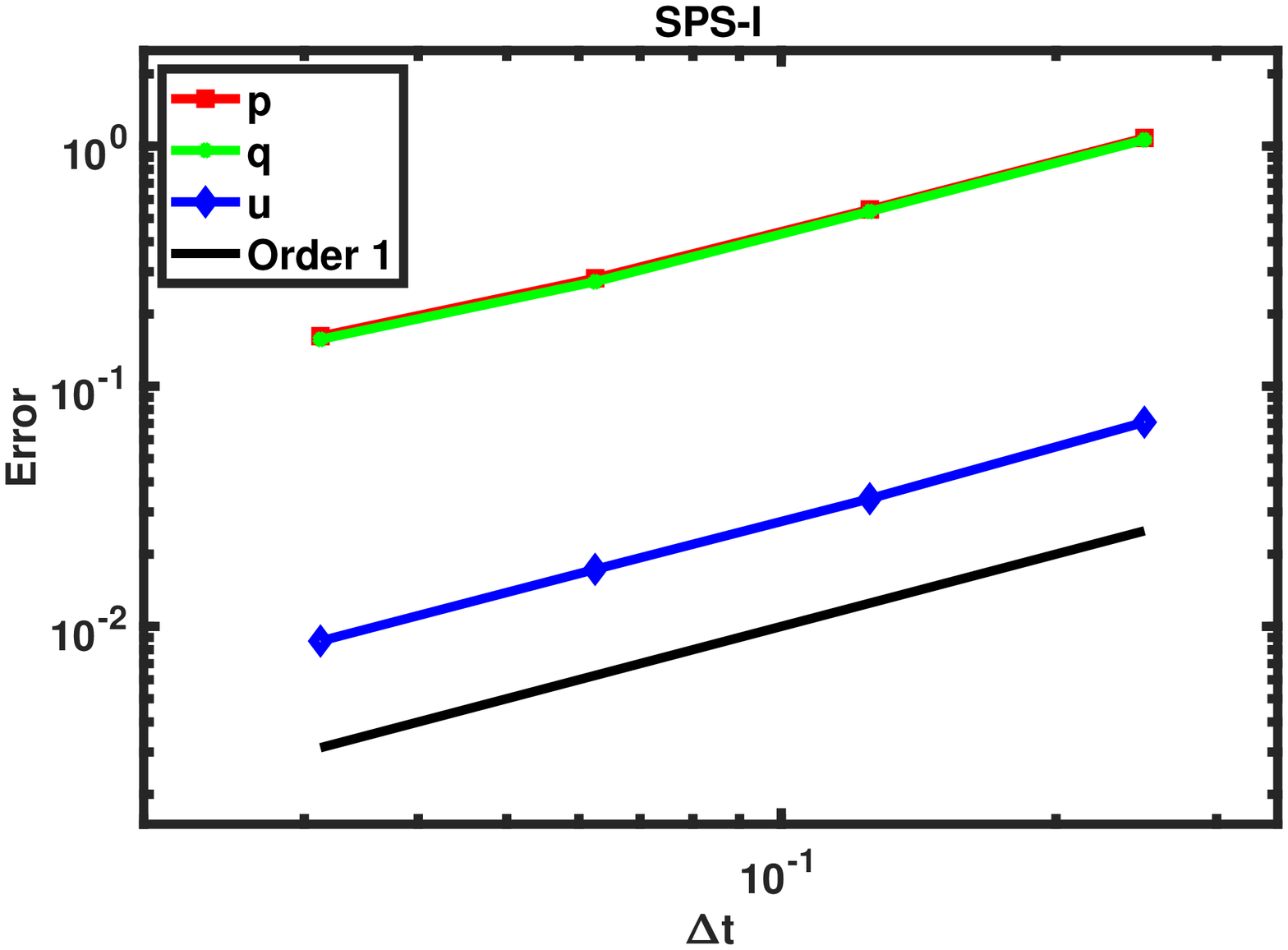}
			\includegraphics[height=4cm,width=4.5cm]{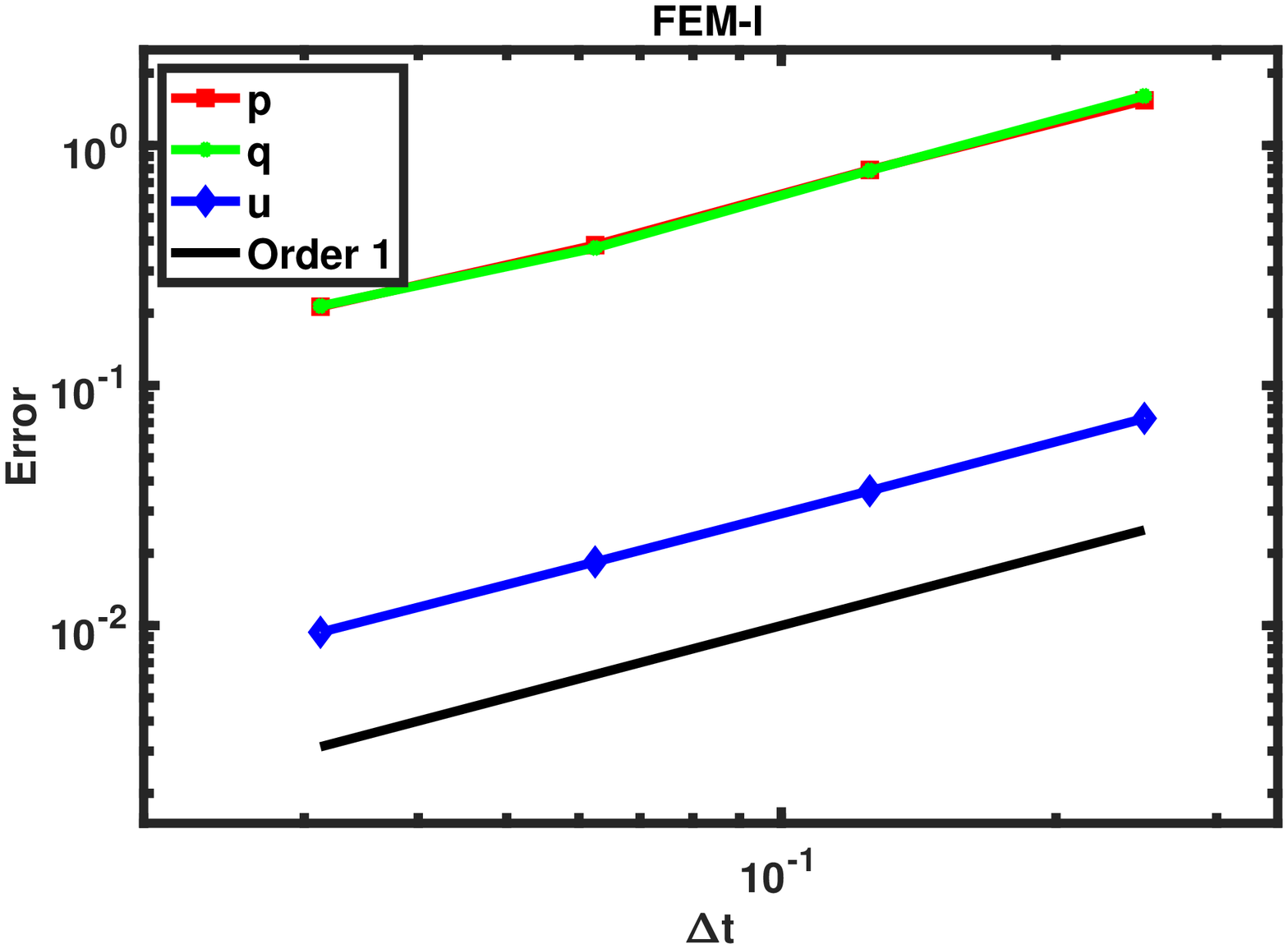}
		\end{minipage}
	}
	\caption{Mean-square errors of fully-discrete schemes in time with T=1.} 
	\label{fig1}
\end{figure}

\begin{figure}[h]
	\centering
	\subfigure{
		\begin{minipage}{15cm}
			\centering
			\includegraphics[height=4cm,width=4.5cm]{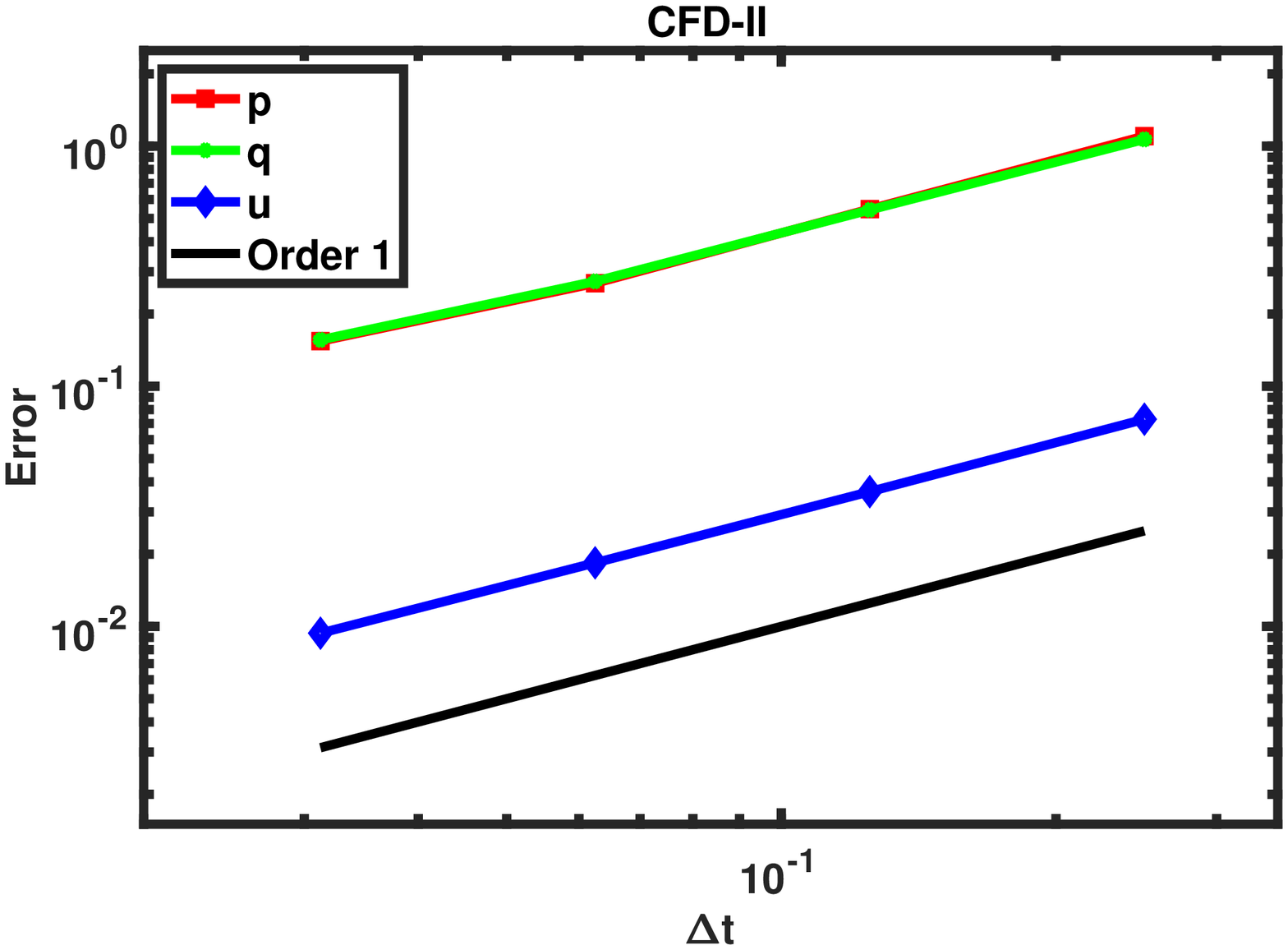}
			\includegraphics[height=4cm,width=4.5cm]{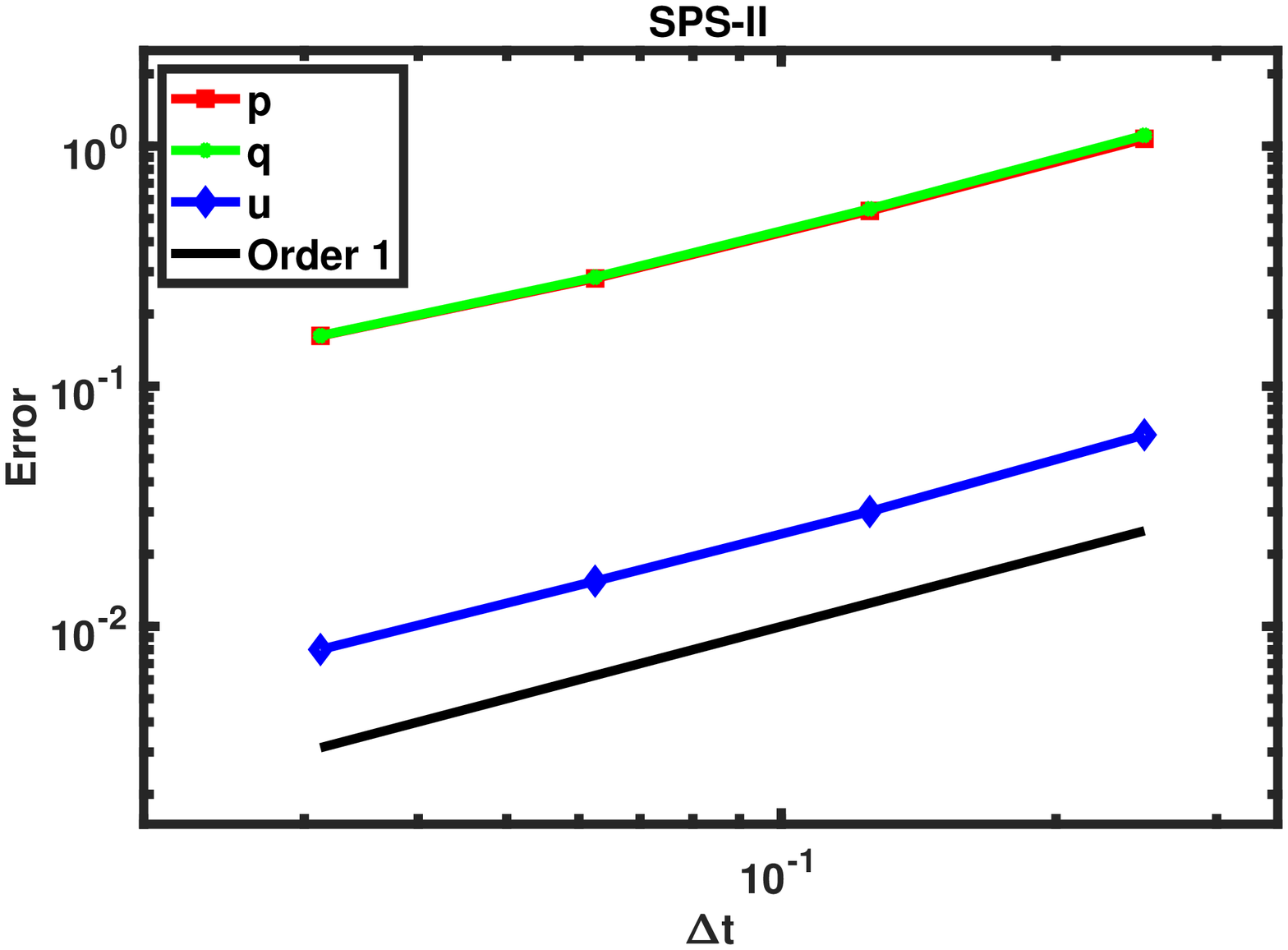}
			\includegraphics[height=4cm,width=4.5cm]{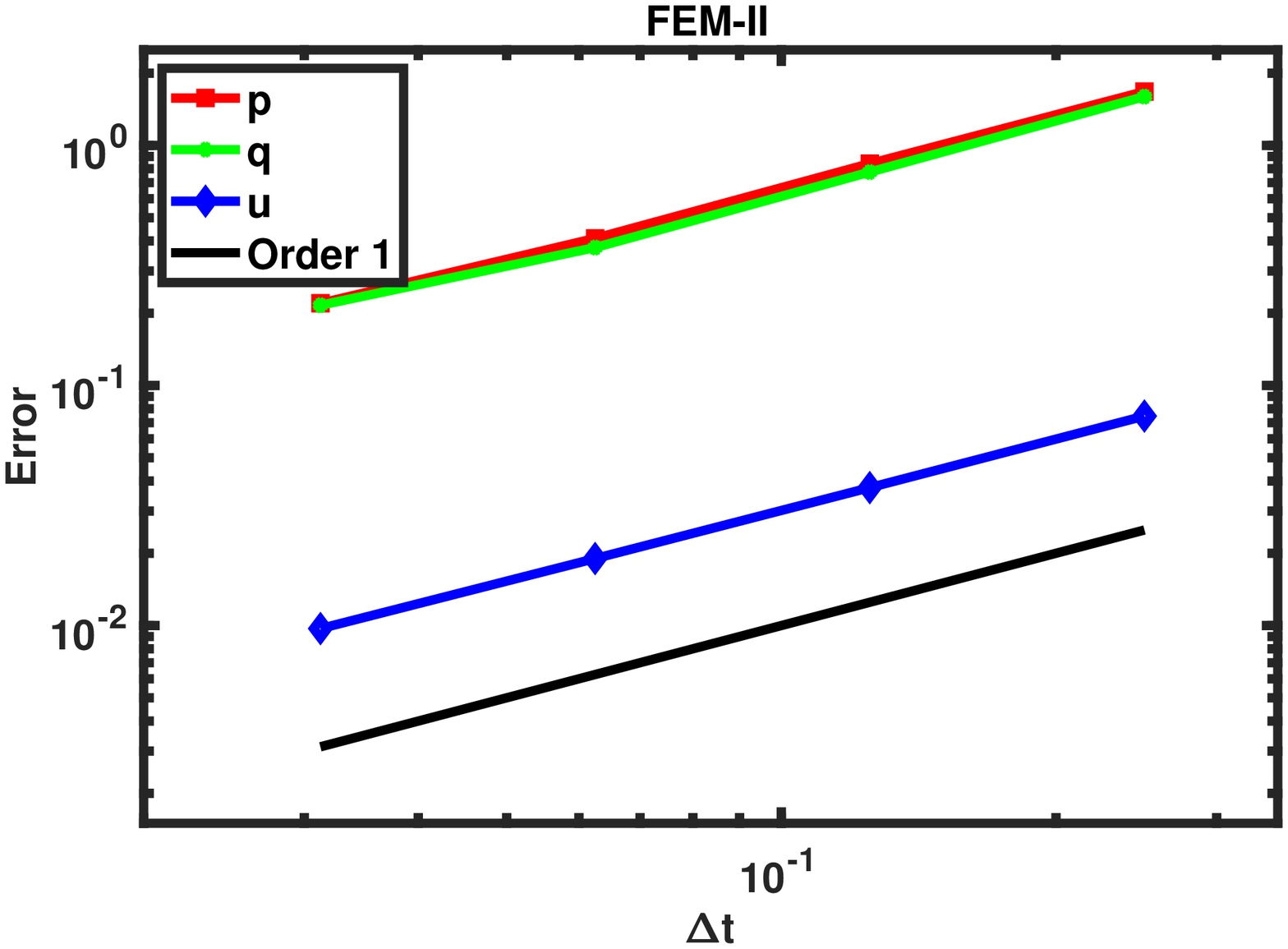}
		\end{minipage}
	}
	\caption{Mean-square errors of fully-discrete schemes in time with T=1.} 
	\label{fig2}
\end{figure}

For the sake of simplicity, we denote the fully-discrete scheme \eqref{SRK} as FD-SRK, which implies that the numerical scheme is based on the finite difference in space and symplectic parametric Runge--Kutta method with $\alpha = 0.001$ and $s=2$ in time. 
Similarly, for the case that central difference, sine pseudo-spectral method, or finite element method is employed in the spatial direction, and the modified finite difference method \eqref{CEFD} or midpoint method \eqref{MSP} is applied in the temporal direction, the corresponding numerical schemes are denoted by CFD-I, SPS-I, FEM-I and CFD-II, SPS-II, FEM-II, respectively. 
Here, we use piecewise linear polynomials for the finite element method. 
Figs. \ref{fig0}-\ref{fig2} show the mean-square error for seven fully-discrete schemes against $\Delta t=2^{-s}, s=3, 4, 5, 6$ on log-log scale at time $T = 1,$ when $a = -15, b=15$, and the initial conditions are
\begin{align*}
& \varphi( 0)=\frac{3 \sqrt{2}}{4 \sqrt{1-\theta^2}} \operatorname{sech}^2\left( \frac{x}{2 \sqrt{1-\theta^2}}\right) \exp \left(i\theta x\right), \\
& u( 0)=\frac{3}{4\left(1-\theta^2\right)} \operatorname{sech}^2\left( \frac{x}{2 \sqrt{1-\theta^2}}\right),\\
& u_t( 0)=\frac{3\theta}{4\left(1-\theta^2\right)^{3/2}} \operatorname{sech}^2\left( \frac{x}{2 \sqrt{1-\theta^2}}\right)\operatorname{tanh}\left( \frac{x}{2 \sqrt{1-\theta^2}}\right),
\end{align*}
with $\theta =0.3$. 
The exact solution is computed by implementing the proposed numerical schemes with a small time step size $\Delta t=2^{-8}$ and small space step size $h= 15\times 2^{-7}$. 
It can be observed that the slopes of seven fully-discrete schemes are close to 1 on the temporal convergence order. 
The theoretical result will be studied in future work.

\begin{figure}[h]
	\centering
	\subfigure{
		\begin{minipage}{15cm}
			\centering
			\includegraphics[height=4cm,width=4.5cm]{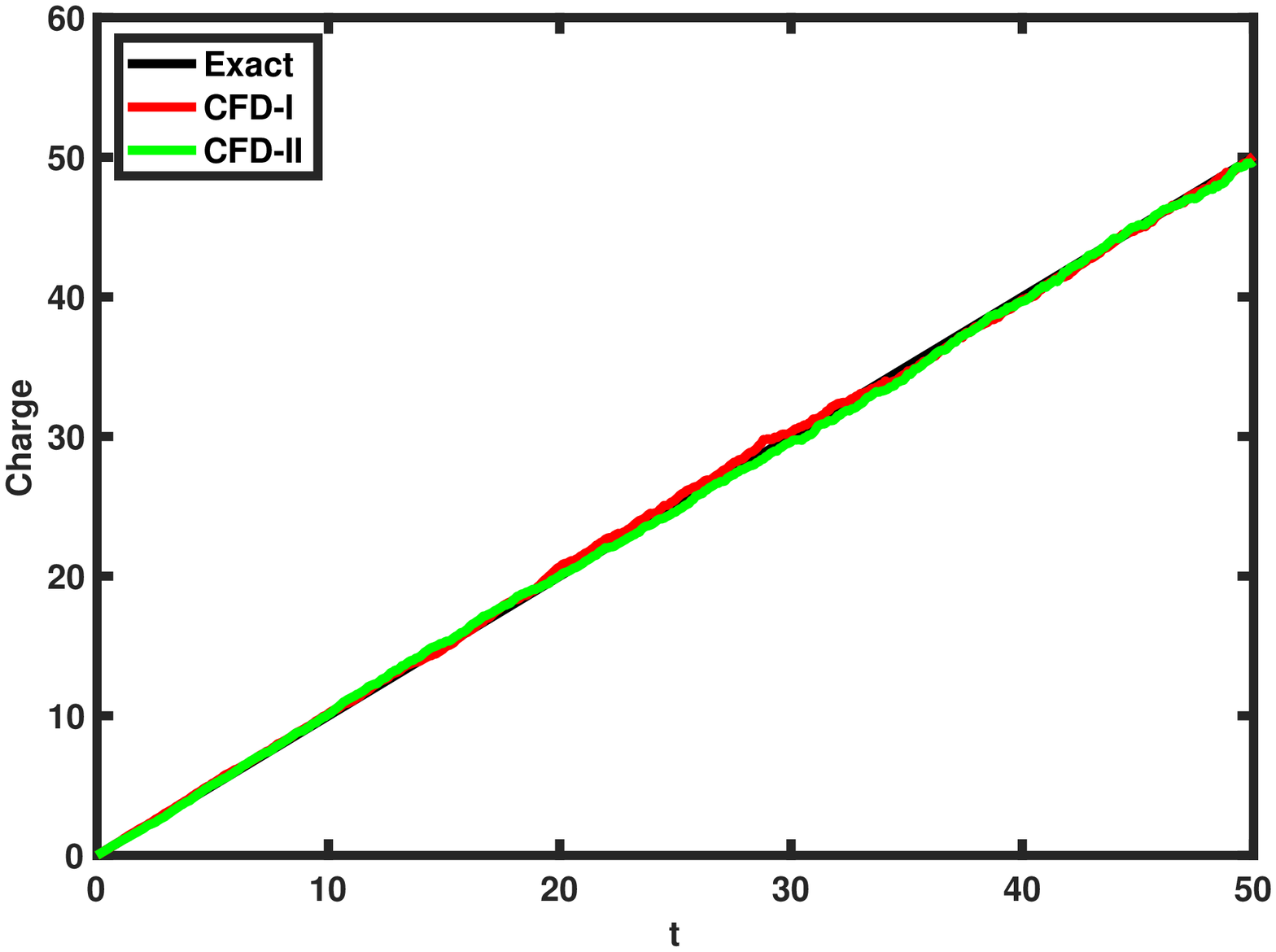}
			\includegraphics[height=4cm,width=4.5cm]{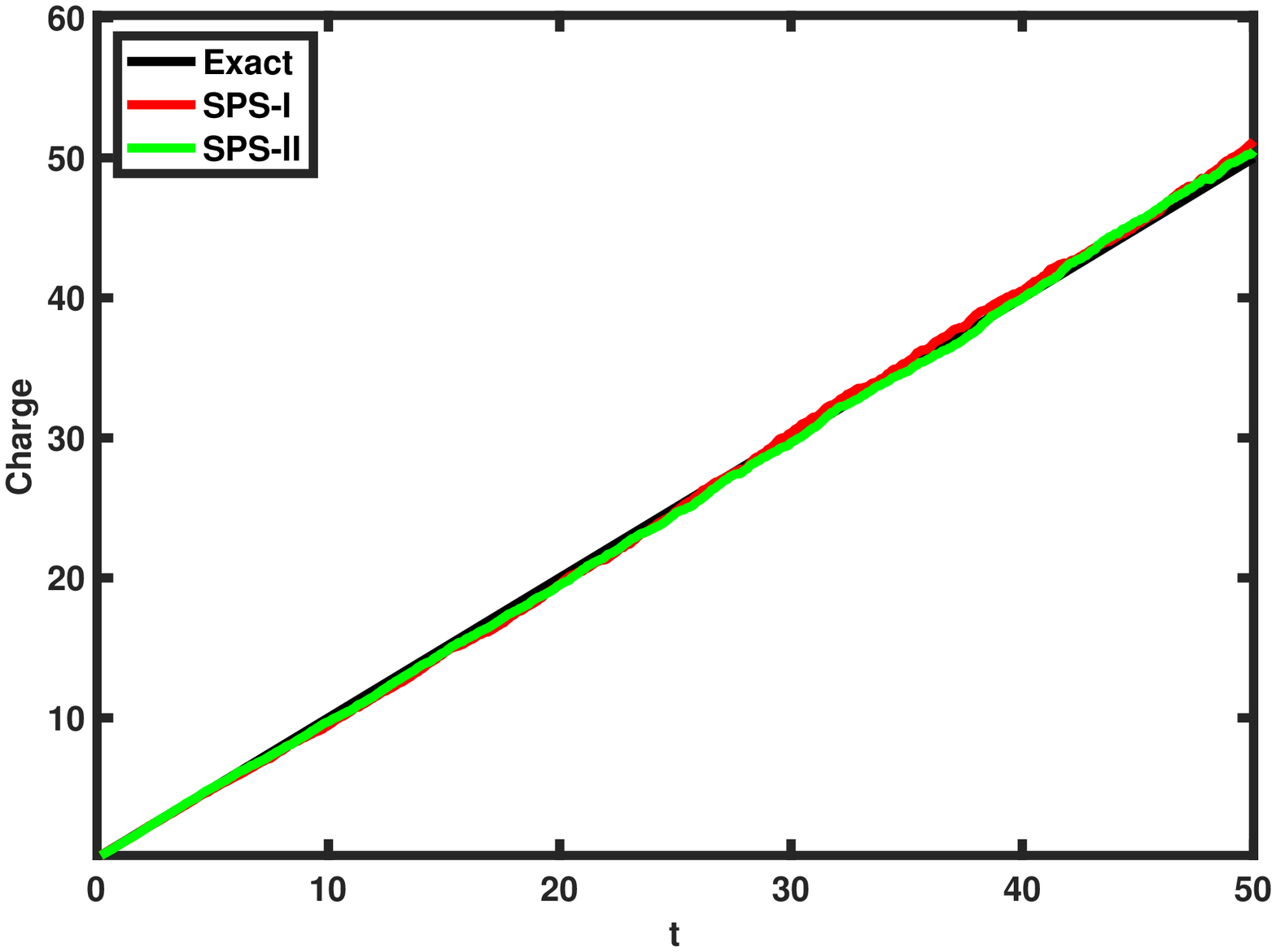}
			\includegraphics[height=4cm,width=4.5cm]{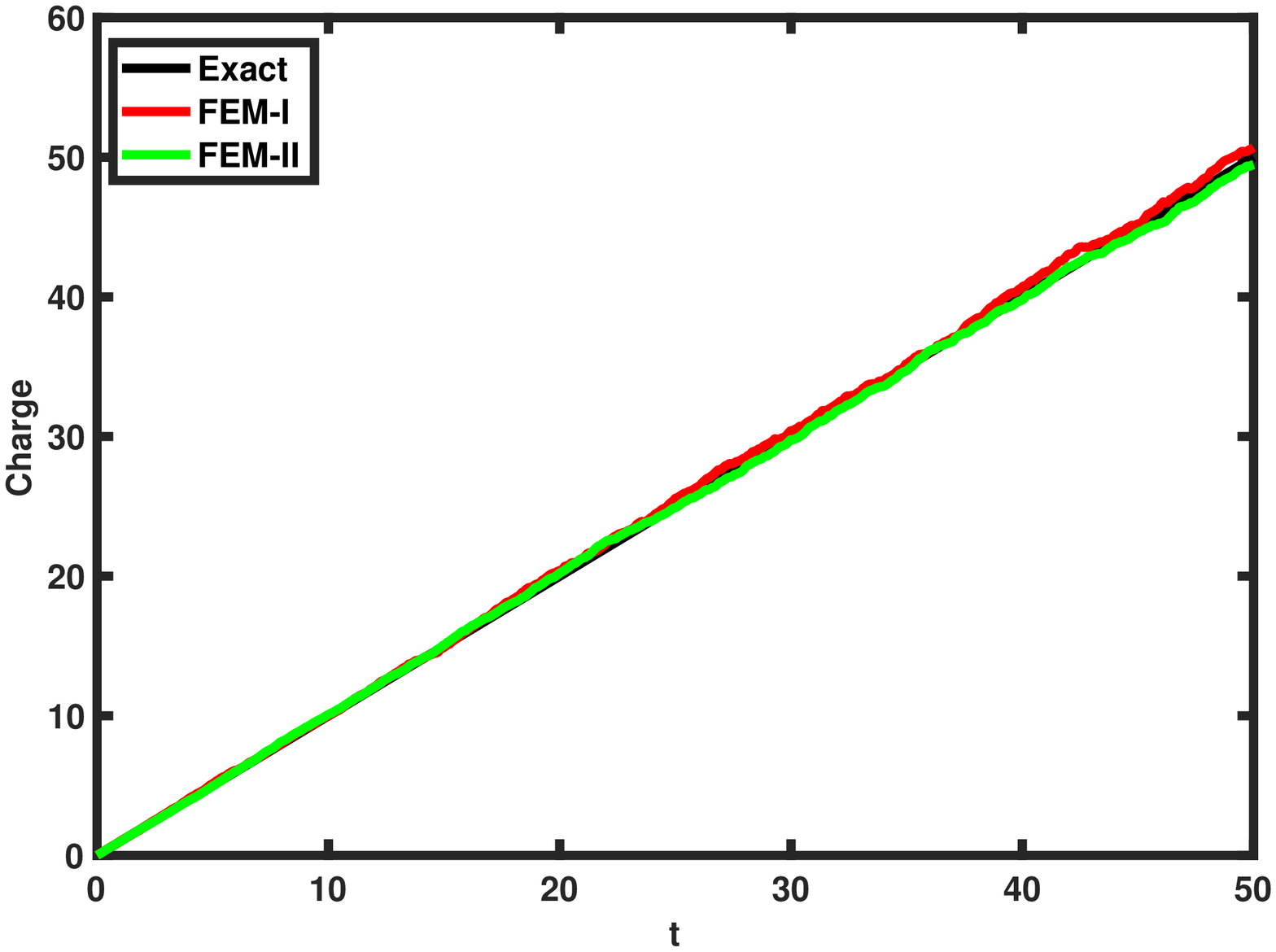}
		\end{minipage}
	}
	\caption{Averaged charge evolution relationship with $\Delta t = 25/2^8, h=1/2^4.$} 
	\label{fig3}
\end{figure}

\begin{figure}[h]
	\centering
	\subfigure{
		\begin{minipage}{15cm}
			\centering
			\includegraphics[height=4cm,width=4.5cm]{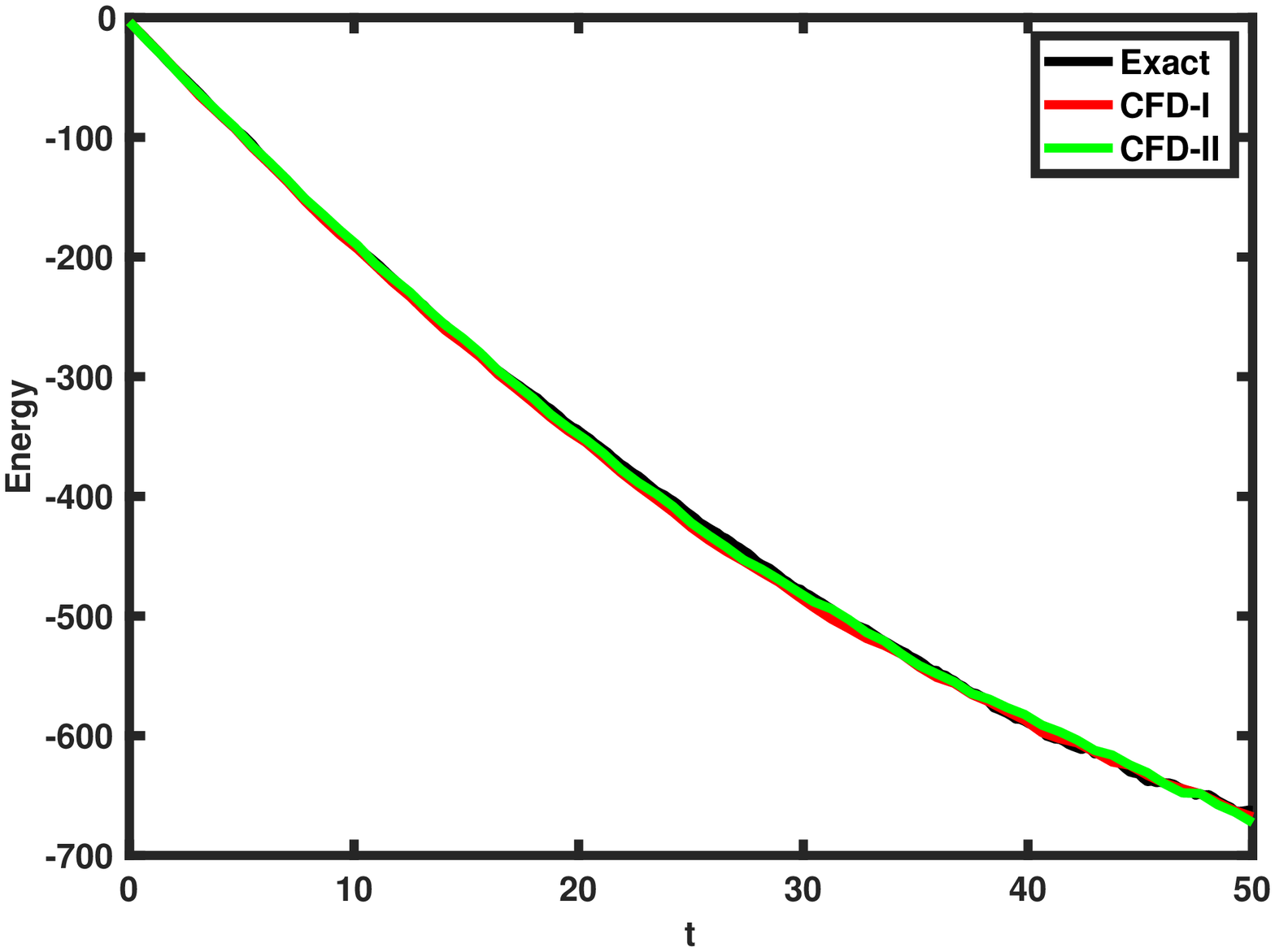}
			\includegraphics[height=4cm,width=4.5cm]{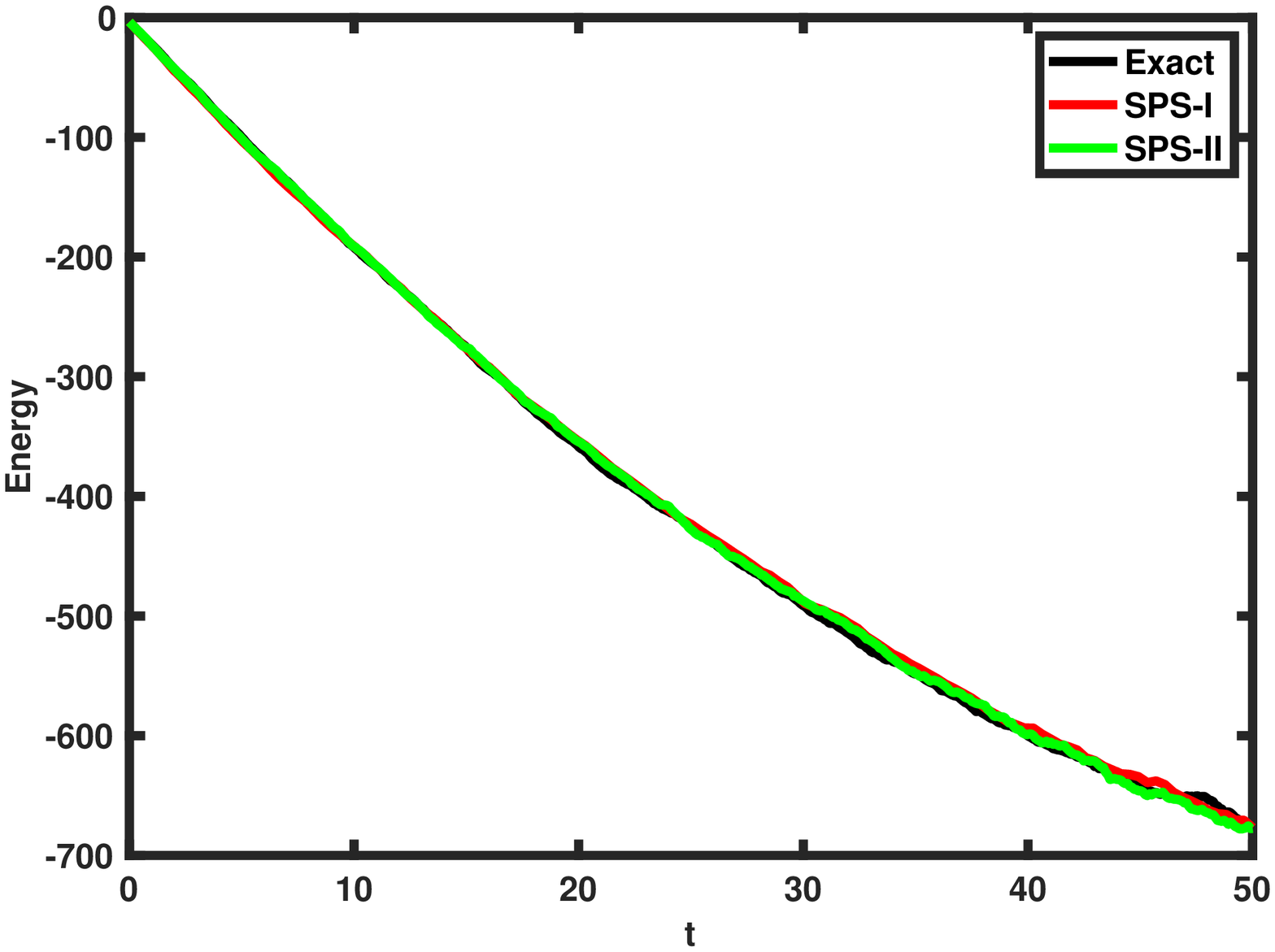}
			\includegraphics[height=4cm,width=4.5cm]{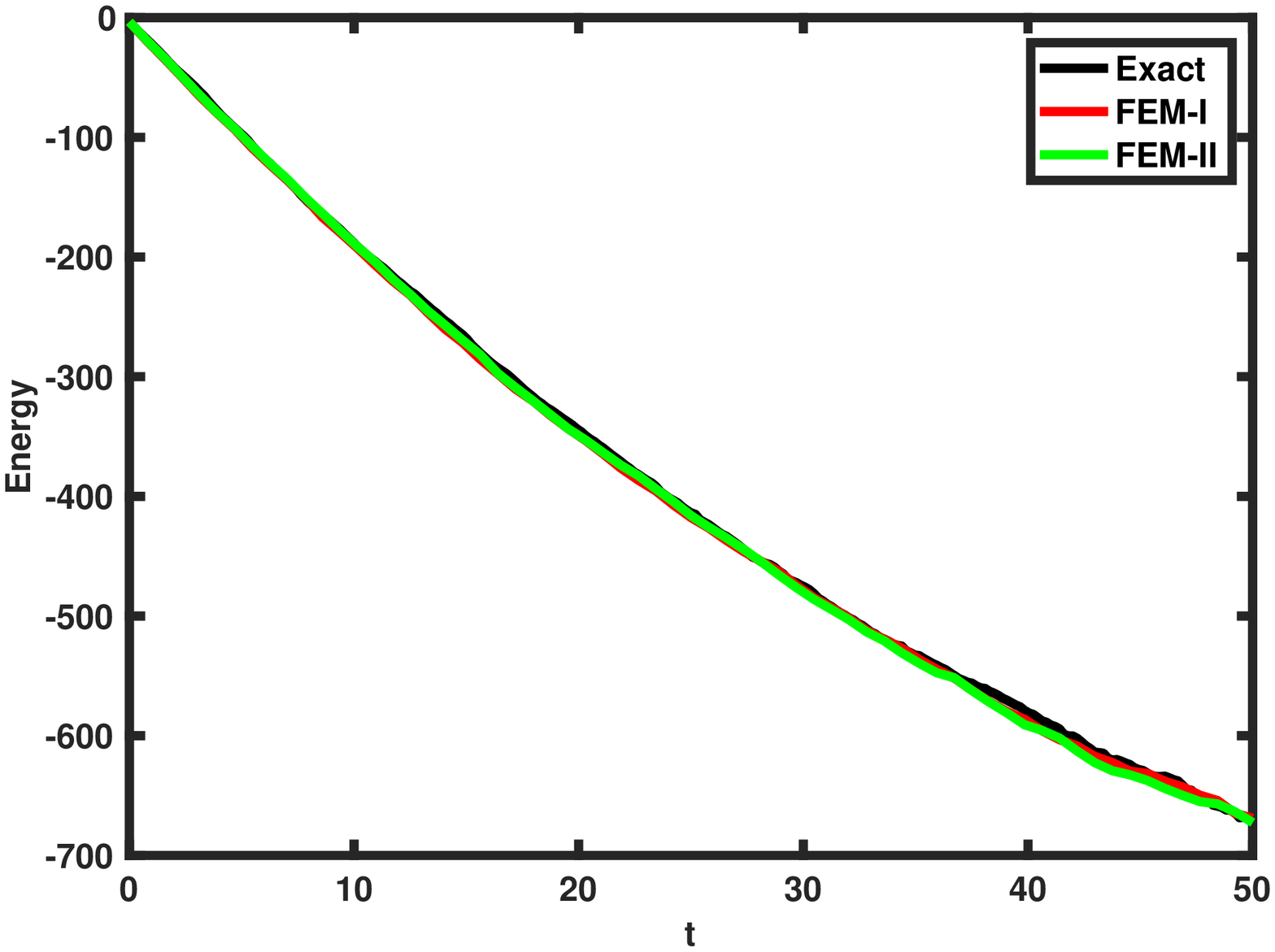}
		\end{minipage}
	}
	\caption{Averaged energy evolution relationship with $\Delta t = 25/2^{10}, h=1/2^3$.} 
	\label{fig4}
\end{figure}

When testing the long-time behaviors for proposed schemes, we choose $a=0,$ $b=1,$ $T=50$ and $\varphi( 0)= 0, u( 0) = 0, u_t( 0)=1$ as the initial conditions. 
The reference line (black line) in Figs. \ref{fig3}-\ref{fig4} stands for the averaged charge evolution law and averaged energy evolution law of the exact solution, respectively. 
It can be observed that the proposed schemes named CFD-I, SPS-I, FEM-I, CFD-II, SPS-II, and FEM-II reproduce the linear growth of the averaged charge and the evolution of the averaged energy. 
It implies that the proposed schemes preserve perfectly both the averaged charge and energy evolution law.

\section{Conclusion}
In this paper, novel structure-preserving schemes are proposed for solving stochastic KGS equations with additive noise. 
We prove that the fully-discrete scheme based on central difference, sine pseudo-spectral method, or finite element method in space and the finite difference method in time, preserves the averaged charge and energy evolution law. 
Besides, we propose a class of multi-symplectic methods through finite difference method in space and symplectic Runge--Kutta method in time. 
Compared with the classical Runge--Kutta method, the proposed multi-symplectic method is more flexible due to the flexibility of the parameter $\alpha$. 
In reality, there are still many important and challenging problems that remain to be solved, such as, studying the strong convergence analysis and estimating the strong convergence order for the proposed schemes; constructing ergodic fully-discrete scheme for damped stochastic KGS equations. We will investigate these problems in our near future.

\section*{Acknowledgements}
This work is supported by National Natural Science Foundation of China (No. 12101596, No. 12031020,  No.12171047).


\end{document}